\documentclass[a4paper,12pt]{article}
\usepackage[cp1251]{inputenc}
\usepackage[russian]{babel}
\usepackage{amsfonts, amssymb, amsmath, amsthm, amscd}
\usepackage{cite}
\author{A.A. Vasil'eva}
\title{Estimates for entropy numbers of embedding operators of
function spaces on sets with tree-like structure: some limiting
cases}
\date{}
\begin{document}

\maketitle

\newenvironment{Biblio}{%
                  \renewcommand{\refname}{\footnotesize REFERENCES}%
                  \begin{thebibliography}{99}%
                  \itemsep -0.2em}%
                  {\end{thebibliography}}

\def\inff{\mathop{\smash\inf\vphantom\sup}}
\renewcommand{\le}{\leqslant}
\renewcommand{\ge}{\geqslant}
\newcommand{\sgn}{\mathrm {sgn}\,}
\newcommand{\inter}{\mathrm {int}\,}
\newcommand{\dist}{\mathrm {dist}}
\newcommand{\supp}{\mathrm {supp}\,}
\newcommand{\R}{\mathbb{R}}
\renewcommand{\C}{\mathbb{C}}
\newcommand{\Z}{\mathbb{Z}}
\newcommand{\N}{\mathbb{N}}
\newcommand{\Q}{\mathbb{Q}}
\theoremstyle{plain}
\newtheorem{Trm}{Theorem}
\newtheorem{trma}{Theorem}
\newtheorem{Def}{Definition}
\newtheorem{Cor}{Corollary}
\newtheorem{Lem}{Lemma}
\newtheorem{Rem}{Remark}
\newtheorem{Sta}{Proposition}
\newtheorem{Sup}{Assumption}
\newtheorem{Supp}{Assumption}
\newtheorem{Exa}{Example}
\renewcommand{\proofname}{\bf Proof}
\renewcommand{\thetrma}{\Alph{trma}}
\renewcommand{\theSupp}{\Alph{Supp}}

\section{Introduction}

In \cite{vas_entr} order estimates for entropy numbers of the
embedding operator of a weighted Sobolev space on a John domain
into a weighted Lebesgue space were obtained, as well as estimates
for entropy numbers of a two-weighted summation operator on a
tree. Here we consider some critical cases.

Recall the definition of entropy numbers (see, e.g.,
\cite{piet_op, carl_steph, edm_trieb_book}).

\begin{Def}
Let $X$, $Y$ be normed spaces, and let $T:X\rightarrow Y$ be a
linear continuous operator. Entropy numbers of $T$ are defined by
$$
e_k(T)=\inf \left\{\varepsilon>0:\; \exists y_1, \, \dots, \,
y_{2^{k-1}}\in Y: \; T(B_X) \subset \cup_{i=1}^{2^{k-1}}
(y_i+\varepsilon B_Y) \right\}, \quad k\in \N.
$$
\end{Def}

Kolmogorov, Tikhomirov, Birman and Solomyak \cite{kolm_tikh1,
tikh_entr, birm} studied properties of $\varepsilon$-entropy (this
magnitude is related to entropy numbers of embedding operators).

Denote by $l_p^m$ $(1\le p\le \infty)$ the space $\R^m$ with norm
$$
\|(x_1, \, \dots , \, x_m)\| _{l_p^m}= \left\{
\begin{array}{l}(| x_1 | ^p+\dots+ | x_m | ^p)^{1/p}\text{ for
}p<\infty ,\\ \max \{| x_1 | , \, \dots, \, | x_m |\}\text{ for
}p=\infty .\end{array}\right .
$$
Estimates for entropy numbers of the embedding operator of $l_p^m$
into $l_q^m$ were obtained in the paper of Sch\"{u}tt
\cite{c_schutt} (see also \cite{edm_trieb_book}). Later Edmunds
and Netrusov \cite{edm_netr1}, \cite{edm_netr2} generalized this
result for vector-valued sequence spaces (in particular, for
sequence spaces with mixed norm).

Haroske, Triebel, K\"{u}hn, Leopold, Sickel and Skrzypczak
\cite{kuhn4, kuhn5, kuhn_01, kuhn_01_g, kuhn_05, kuhn_08,
kuhn_leopold, kuhn_tr_mian, har94_1, har94_2, haroske, haroske2,
haroske3, har_tr05} studied the problem of estimating entropy
numbers of embeddings of weighted sequence spaces or weighted
Besov and Triebel--Lizorkin spaces.

Triebel \cite{tr_jat} and Mieth \cite{mieth_15} studied the
problem of estimating entropy numbers of embedding operators of
weighted Sobolev spaces on a ball with weights that have
singularity at the origin.

Lifshits and Linde \cite{lif_linde} obtained estimates for entropy
numbers of two-weighted Hardy-type operators on a semiaxis (under
some conditions on weights). The similar problem for one-weighted
Riemann-Liouville operators was considered in the paper of
Lomakina and Stepanov \cite{step_lom}. In addition, Lifshits and
Linde \cite{lifs_m, l_l, l_l1} studied the problem of estimating
entropy numbers of two-weighted summation operators on a tree.
%${\cal T}$ from the space $l_p({\cal T})$ into $l_\infty({\cal
%T})$, or of its dual from $l_1({\cal T})$ into $l_{p'}({\cal T})$
%(here $p'=\frac{p}{p-1}$). Notice that here we consider for
%$1<p<q<\infty$ the critical case similar to the case \cite{lifs_m,
%l_l1}.

This paper is organized as follows. In \S 2 we obtain the general
result about upper estimates for entropy numbers of embedding
operators of function spaces on a set with tree-like structure.
The properties of such spaces are almost the same as properties of
function spaces defined in \cite{vas_width_raspr, vas_entr} (see
Assumptions \ref{sup1}, \ref{sup2}, \ref{sup3}), but there are
some differences. In particular, here we suppose that
(\ref{f_pom_f}) holds; this condition cannot be directly derived
from known results for weighted Sobolev and Lebesgue spaces.
Therefore, first we consider some particular cases of function
spaces satisfying Assumptions \ref{supp1}--\ref{supp3} or
\ref{supp1}, \ref{supp2}, \ref{supp4} (see \S 3) and we prove that
these spaces satisfy Assumptions \ref{sup1}--\ref{sup3}. In \S 4
we obtain order estimates for entropy numbers of embedding
operators of weighted Sobolev spaces; to this end, we prove that
under given conditions on weights Assumptions
\ref{supp1}--\ref{supp3} or \ref{supp1}, \ref{supp2}, \ref{supp4}
hold. In \S 5 we obtain order estimates for entropy numbers of
two-weighted summation operators on a tree in critical cases.

\section{Upper estimates for entropy numbers of embedding operators of
function spaces on sets with tree-like structure}

Let us give some notations.

Let $(\Omega, \, \Sigma, \, {\rm mes})$ be a measure space. We say
that sets $A$, $B\subset \Omega$ are disjoint if ${\rm mes}(A\cap
B)=0$. Let $E$, $E_1, \, \dots, \, E_m\subset \Omega$ be
measurable sets, and let $m\in \N\cup \{\infty\}$. We say that
$\{E_i\}_{i=1}^m$ is a partition of $E$ if the sets $E_i$ are
pairwise disjoint and ${\rm mes}\left(\left(\cup _{i=1}^m
E_i\right)\bigtriangleup E\right)=0$.

Denote by $\chi_E(\cdot)$ the indicator function of a set $E$.

Let ${\cal G}$ be a graph containing at most countable number of
vertices. We shall denote by ${\bf V}({\cal G})$ the vertex set of
${\cal G}$. Two vertices are called {\it adjacent} if there is an
edge between them. Let $\xi_i\in {\bf V}({\cal G})$, $1\le i\le
n$. The sequence $(\xi_1, \, \dots, \, \xi_n)$ is called a {\it
path} if the vertices $\xi_i$ and $\xi_{i+1}$ are adjacent for any
$i=1, \, \dots , \, n-1$. If all the vertices $\xi_i$ are
distinct, then such a path is called {\it simple}.

Let $({\cal T}, \, \xi_0)$ be a tree with a distinguished vertex
(or a root) $\xi_0$. We introduce a partial order on ${\bf
V}({\cal T})$ as follows: we say that $\xi'>\xi$ if there exists a
simple path $(\xi_0, \, \xi_1, \, \dots , \, \xi_n, \, \xi')$ such
that $\xi=\xi_k$ for some $k\in \overline{0, \, n}$. In this case,
we set $\rho_{{\cal T}}(\xi, \, \xi')=\rho_{{\cal T}}(\xi', \,
\xi) =n+1-k$. In addition, we denote $\rho_{{\cal T}}(\xi, \,
\xi)=0$. If $\xi'>\xi$ or $\xi'=\xi$, then we write $\xi'\ge \xi$.
This partial order on ${\cal T}$ induces a partial order on its
subtree.

Let ${\cal G}$ be a disjoint union of trees $({\cal T}_j, \,
\xi_j)$, $1\le j\le k$. Then the partial order on each tree ${\cal
T}_j$ induces the partial order on ${\cal G}$.

Given  $j\in \Z_+$, $\xi\in {\bf V}({\cal T})$, we denote
$$
\label{v1v}{\bf V}_j(\xi):={\bf V}_j ^{{\cal T}}(\xi):=
\{\xi'\ge\xi:\; \rho_{{\cal T}}(\xi, \, \xi')=j\}.
$$
For $\xi\in {\bf V}({\cal T})$ we denote by ${\cal T}_\xi=({\cal
T}_\xi, \, \xi)$ the subtree in ${\cal T}$ with vertex set
$$
\{\xi'\in {\bf V}({\cal T}):\xi'\ge \xi\}.
$$

Let ${\cal G}$ be a subgraph in ${\cal T}$. Denote by ${\bf
V}_{\max} ({\cal G})$ and ${\bf V}_{\min}({\cal G})$ the sets of
maximal and minimal vertices in ${\cal G}$, respectively.

Let ${\bf W}\subset {\bf V}({\cal T})$. We say that ${\cal
G}\subset {\cal T}$ is a maximal subgraph on the vertex set ${\bf
W}$ if ${\bf V}({\cal G})={\bf W}$ and any two vertices $\xi'$,
$\xi''\in {\bf W}$ adjacent in ${\cal T}$ are also adjacent in
${\cal G}$. Given subgraphs $\Gamma_1$, $\Gamma_2 \subset {\cal
T}$, we denote by $\Gamma_1 \cap \Gamma_2$ the maximal subgraph in
${\cal T}$ on the vertex set ${\bf V}(\Gamma_1) \cap {\bf
V}(\Gamma_2)$.

Let ${\bf P}=\{{\cal T}_j\}_{j\in \N}$ be a family of subtrees in
${\cal T}$ such that ${\bf V}({\cal T}_j)\cap {\bf V}({\cal
T}_{j'}) =\varnothing$ for $j\ne j'$ and $\cup _{j\in \N} {\bf
V}({\cal T}_j) ={\bf V}({\cal T})$. Then $\{{\cal T}_j\} _{j\in
\N}$ is called a partition of the tree ${\cal T}$. Let $\xi_j$ be
the minimal vertex of ${\cal T}_j$. We say that the tree ${\cal
T}_s$ succeeds the tree ${\cal T}_j$ (or ${\cal T}_j$ precedes the
tree ${\cal T}_s$) if $\xi_j<\xi_s$ and
$$\{\xi\in {\cal T}:\; \xi_j\le \xi<\xi_s\} \subset {\bf V}({\cal
T}_j).$$ If $\Gamma \subset {\cal T}$ is the maximal subgraph on
the vertex set ${\bf W}$, we set ${\bf P}|_{\Gamma} =\{\Gamma \cap
{\cal T}_j\}_{j\in \N}$.

We consider the function spaces on sets with tree-like structure
from \cite{vas_width_raspr, vas_entr}.

\label{xpyq}Let $(\Omega, \, \Sigma, \, {\rm mes})$ be a measure
space, let $\hat\Theta$ be a countable partition of $\Omega$ into
measurable subsets, let $({\cal A}, \, \xi_0)$ be a tree such that
\begin{align}
\label{c_v1_a} \exists c_1\ge 1:\quad {\rm card}\, {\bf
V}_1^{{\cal A}}(\xi)\le c_1, \quad \xi \in {\bf V}({\cal A}),
\end{align}
and let $\hat F:{\bf V}({\cal A}) \rightarrow \hat\Theta$ be a
bijective mapping.

Throughout we consider at most countable partitions into
measurable subsets.

Let $1\le p, \, q< \infty$ be arbitrary numbers. We suppose that,
for any measurable subset $E\subset \Omega$, the following spaces
are defined:
\begin{itemize}
\item the space $X_p(E)$ with seminorm $\|\cdot\|_{X_p(E)}$,
\item the Banach space $Y_q(E)$ with norm $\|\cdot\|_{Y_q(E)}$,
\end{itemize}
which all satisfy the following conditions:
\begin{enumerate}
\item $X_p(\Omega)\subset Y_q(\Omega)$;
\item $X_p(E)=\{f|_E:\; f\in X_p(\Omega)\}$, $Y_q(E)=\{f|_E:\; f\in
Y_q(\Omega)\}$;
\item if ${\rm mes}\, E=0$, then $\dim \, Y_q(E)=\dim \, X_p(E)=0$;
\item if $E\subset \Omega$, $E_j\subset \Omega$ ($j\in \N$)
are measurable subsets, $E=\sqcup _{j\in \N} E_j$, then
\begin{align}
\label{f_xp} \|f\|_{X_p(E)}=\left\| \bigl\{
\|f|_{E_j}\|_{X_p(E_j)}\bigr\}_{j\in \N}\right\|_{l_p},\quad f\in
X_p(E),
\end{align}
\begin{align}
\label{f_yq} \|f\|_{Y_q(E)}=\left\| \bigl\{\|f|_{E_j}\|
_{Y_q(E_j)}\bigr\}_{j\in \N}\right\|_{l_q}, \quad f\in Y_q(E);
\end{align}
\item if $E\in \Sigma$, $f\in Y_q(\Omega)$, then $f\cdot \chi_E\in
Y_q(\Omega)$.
\end{enumerate}

Let ${\cal P}(\Omega)\subset X_p(\Omega)$  be a subspace of finite
dimension $r_0$ and let $\|f\|_{X_p(\Omega)}=0$ for any $f\in
{\cal P}(\Omega)$. For each measurable subset $E\subset \Omega$ we
write ${\cal P}(E)=\{P|_E:\; P\in {\cal P}(\Omega)\}$. Let
$G\subset \Omega$ be a measurable subset and let $T$ be a
partition of $G$. We set
\begin{align}
\label{st_omega} {\cal S}_{T}(\Omega)=\{f:\Omega\rightarrow \R:\,
f|_E\in {\cal P}(E), \; f|_{\Omega\backslash G}=0\}.
\end{align}
If $T$ is finite, then ${\cal S}_{T}(\Omega)\subset Y_q(\Omega)$
(see property 5).

For any finite partition $T=\{E_j\}_{j=1}^n$ of the set $E$ and
for each function $f\in Y_q(\Omega)$ we put
\begin{align}
\label{fpqt} \|f\|_{p,q,T}=\left(\sum \limits _{j=1}^n
\|f|_{E_j}\|_{Y_q(E_j)}
^{\sigma_{p,q}}\right)^{\frac{1}{\sigma_{p,q}}}
\end{align}
with $\sigma_{p,q}=\min\{p, \, q\}$.  Denote by $Y_{p,q,T}(E)$ the
space $Y_q(E)$ with the norm $\|\cdot\|_{p,q,T}$. Notice that
$\|\cdot\| _{Y_q(E)}\le \|\cdot\|_{p,q,T}$.

For each subtree ${\cal A}'\subset {\cal A}$ we set $$\Omega
_{{\cal A}'}=\cup _{\xi\in {\bf V}({\cal A}')} \hat F(\xi).$$

In \cite{vas_entr} upper estimates for entropy numbers of the
embedding operator of the space $\hat X_p(\Omega)$ into
$Y_q(\Omega)$ were obtained under some conditions on these spaces
(the space $\hat X_p(\Omega)\cong X_p(\Omega)/ {\cal P}(\Omega)$
will be defined later). Some limiting relations between the
parameters were not considered. Here we investigate one of those
critical cases.

Throughout we assume that $1<p<q<\infty$ and the following
conditions hold.

\begin{Sup}
\label{sup1} There exist a partition $\{{\cal A}_{t,i}\}_{t\ge
t_0, \, i\in \hat J_t}$ of the tree ${\cal A}$ and a number $c\ge
1$ with the following properties.
\begin{enumerate}
\item If the tree ${\cal A}_{t',i'}$ follows the tree ${\cal
A}_{t,i}$, then $t'=t+1$.
\item For each vertex $\xi_*\in {\bf V}({\cal A})$ there exists a
linear continuous projection $P_{\xi_*}:Y_q(\Omega)\rightarrow
{\cal P}(\Omega)$ such that for any function $f\in X_p(\Omega)$
and for any subtree ${\cal D}\subset {\cal A}$ rooted at $\xi_*$
\begin{align}
\label{f_pom_f} \|f-P_{\xi_*}f\|_{Y_q(\Omega_{{\cal D}})}^q \le c
\sum \limits _{t=t_0}^\infty \sum \limits _{i\in J_{t,{\cal D}}}
2^{\left(1-\frac{q}{p}\right)t} \|f\|^q_{X_p(\Omega _{{\cal
D}_{t,i}})};
\end{align}
here
\begin{align}
\label{jtd} J_{t,{\cal D}} =\{i\in \hat J_t:\; {\bf V}({\cal
A}_{t,i}) \cap {\bf V}({\cal D}) \ne \varnothing\}, \quad {\cal
D}_{t,i} ={\cal A}_{t,i} \cap {\cal D}.
\end{align}
\end{enumerate}
\end{Sup}

\begin{Sup}
\label{sup2} There exist numbers $\delta_*>0$ and $c_2\ge 1$ such
that for each vertex $\xi\in {\bf V}({\cal A})$ and for any $n\in
\N$, $m\in \Z_+$ there exists a partition $T_{m,n}(G)$ of the set
$G=\hat F(\xi)$ with the following properties:
\begin{enumerate}
\item ${\rm card}\, T_{m,n}(G)\le c_2\cdot 2^mn$.
\item For any $E\in T_{m,n}(G)$ there exists a linear continuous operator
$P_E:Y_q(\Omega)\rightarrow {\cal P}(E)$ such that for any
function $f\in X_p(\Omega)$
\begin{align}
\label{fpef} \|f-P_Ef\|_{Y_q(E)}\le
(2^mn)^{-\delta_*}2^{\left(\frac 1q-\frac 1p\right)t} \|f\|
_{X_p(E)},
\end{align}
where $t\ge t_0$ is such that $\xi \in \cup _{j\in \hat J_t}{\bf
V}({\cal A}_{t,j})$.
\item For any $E\in T_{m,n}(G)$
\begin{align}
\label{ceptm} {\rm card}\,\{E'\in T_{m\pm 1,n}(G):\, {\rm
mes}(E\cap E') >0\} \le c_2.
\end{align}
\end{enumerate}
\end{Sup}

\begin{Sup}
\label{sup3} There exist numbers $\gamma_*>0$, $c_3\ge 1$ and an
absolutely continuous function $\psi_*:(0, \, \infty) \rightarrow
(0, \, \infty)$ such that $\lim \limits _{y\to \infty}
\frac{y\psi_*'(y)}{\psi_*(y)}=0$ and for $\nu_t:= \sum \limits
_{i\in \hat J_t} {\rm card}\, {\bf V}({\cal A}_{t,i})$ the
following estimate holds:
\begin{align}
\label{nu_t_k1} \nu_t\le c_3\cdot 2^{\gamma_*2^{t}}
\psi_*(2^{2^{t}})=: c_3 \overline{\nu}_t,\quad t\ge t_0.
\end{align}
\end{Sup}

Assumption \ref{sup1} together with the inequality $p<q$ implies
that for any $t\ge t_0$ and for each vertex $\xi_*\in {\bf
V}^{\cal A}_{t-t_0}(\xi_0)$ there exists a linear continuous
projection $P_{\xi_*}:Y_q(\Omega)\rightarrow {\cal P}(\Omega)$
such that for any function $f\in X_p(\Omega)$ and for any subtree
${\cal D}\subset {\cal A}$ rooted at $\xi_*$
\begin{align}
\label{f_pom_f_cor} \|f-P_{\xi_*}f\|_{Y_q(\Omega_{{\cal D}})} \le
c \cdot 2^{\left(\frac 1q-\frac{1}{p}\right)t} \|f\|_{X_p(\Omega
_{{\cal D}})}.
\end{align}
Hence, Assumptions 1--3 from \cite{vas_width_raspr, vas_entr} hold
with $\lambda_*=\mu_*=\frac 1p-\frac 1q$ and $u_*\equiv 1$. In
particular, there exist a linear continuous projection $\hat
P:Y_q(\Omega) \rightarrow {\cal P}(\Omega)$ and a number $M>0$
such that for any function $f\in X_p(\Omega)$ the following
estimate holds:
$$\|f-\hat Pf\|_{Y_q(\Omega)} \le M\|f\|_{X_p(\Omega)}.$$ As
$\hat P$ we take the operator $P_{\xi_0}$ (recall that $\xi_0$ is
the root of ${\cal A}$). Similarly as in \cite{vas_entr} we set
$$
\hat X_p(\Omega) =\{f-\hat Pf:\; f\in X_p(\Omega)\}
$$
and denote by $I$ the embedding operator of $\hat X_p(\Omega)$
into $Y_q(\Omega)$.

We set $\mathfrak{Z}_0=(p, \, q, \, c_1, \, c_2,\, c_3,  \, c, \,
\delta_*,\, \gamma_*, \, \psi_*)$.

We use the following notations for order inequalities. Let $X$,
$Y$ be sets, and let $f_1$, $f_2:\ X\times Y\rightarrow
\mathbb{R}_+$. We write $f_1(x, \, y)\underset{y}{\lesssim} f_2(x,
\, y)$ (or $f_2(x, \, y)\underset{y}{\gtrsim} f_1(x, \, y)$) if
for any $y\in Y$ there exists $c(y)>0$ such that $f_1(x, \, y)\le
c(y)f_2(x, \, y)$ for any $x\in X$; $f_1(x, \,
y)\underset{y}{\asymp} f_2(x, \, y)$ if $f_1(x, \, y)
\underset{y}{\lesssim} f_2(x, \, y)$ and $f_2(x, \,
y)\underset{y}{\lesssim} f_1(x, \, y)$.

\begin{Trm}
\label{main_tree} Suppose that Assumptions \ref{sup1}--\ref{sup3}
hold. Then
\begin{align}
\label{main_est} e_n(I:\hat X_p(\Omega) \rightarrow Y_q(\Omega))
\underset{\mathfrak{Z}_0}{\lesssim} n^{\frac 1q-\frac 1p}.
\end{align}
\end{Trm}

Similarly as in \cite{vas_width_raspr}, \cite{vas_entr} we
introduce some more notation.
\begin{itemize}
\item $\hat \xi_{t,i}$ is the minimal vertex of the tree ${\cal
A}_{t,i}$.
\item $G_t=\cup_{\xi\in {\bf V}(\Gamma_t)}\hat F(\xi)=\cup_{i\in \hat
J_t} \Omega_{{\cal A}_{t,i}}$.
\item $\tilde \Gamma_t$ is the maximal subgraph
on the vertex set $\cup _{j\ge t} {\bf V}(\Gamma_j)$, $t\in \N$.
\item $\{\tilde {\cal A}_{t,i}\}_{i\in \overline{J}_t}$
is the set of connected components of the graph $\tilde \Gamma
_t$.
\item $\tilde U_{t,i}=\cup _{\xi\in {\bf V}(\tilde{\cal
A}_{t,i})}\hat F(\xi)$.
\item $\tilde U_t =\cup _{i\in \overline{J}_t} \tilde U_{t,i} =\cup
_{\xi\in {\bf V}(\tilde \Gamma_t)} \hat F(\xi)$.
\end{itemize}
If $t\ge t_0$, then
\begin{align}
\label{v_min_gt} {\bf V}_{\min}(\tilde \Gamma_t)= {\bf
V}_{\min}(\Gamma_t)=\{\hat \xi_{t,i}\}_{i\in \hat J_t}
\end{align}
(see \cite{vas_width_raspr}); hence, we may assume that
\begin{align}
\label{ovrl_it_eq_hat_it} \overline{J}_t=\hat J_t, \quad t\ge t_0.
\end{align}
The set $\hat J_{t_0}$ is a singleton. Denote $\{i_0\}=\hat
J_{t_0}$.

In \cite{vas_entr} the operators $Q_t$ and $P_{t,m}$ were defined
as follows.

{\bf Definition of the operator $Q_t$.} For each $t\ge t_0$, $i\in
\hat J_t\stackrel{(\ref{ovrl_it_eq_hat_it})}{=}\overline{J}_t$
there exists a linear continuous operator $\tilde P_{t,i}:
Y_q(\Omega)\rightarrow {\cal P}(\Omega)$ such that for any
function $f\in \hat X_p(\Omega)$ and for any subtree ${\cal
A}'\subset {\cal A}$ rooted at $\hat \xi_{t,i}$
\begin{align}
\label{ftp_ti_mod} \|f-\tilde P_{t,i}f\|_{Y_q(\Omega_{{\cal A}'})}
\stackrel{(\ref{f_pom_f_cor})}{\underset{\mathfrak{Z}_0}{\lesssim}}
2^{-\left(\frac 1p-\frac 1q\right)t}\|f\| _{X_p(\Omega_{{\cal
A}'})}.
\end{align}
As $\tilde P_{t,i}$ we take $P_{\hat \xi_{t,i}}$. From the
definition of the space $\hat X_p(\Omega)$ and of the operator
$\hat P$ it follows that $\tilde P_{t_0,i_0}|_{\hat
X_p(\Omega)}=0$ (see \cite{vas_entr}).

We set
\begin{align}
\label{qtf_x} \begin{array}{c} Q_tf(x)=\tilde P_{t,i}f(x)=P_{\hat
\xi_{t,i}}f(x) \; \text{for} \; x\in \tilde U_{t,i}, \; i\in
\overline{J}_t, \\ Q_tf(x)=0\; \text{for} \; x\in \Omega
\backslash \tilde U_t, \end{array}
\end{align}
\begin{align}
\label{tt_def} T_t=\{U_{t+1,i}\}_{i\in \overline{J}_{t+1}}.
\end{align}
Since $p< q$, we have for any $f\in B\hat X_p(\Omega)$
\begin{align}
\label{pp3_11} \|f-Q_tf\| _{Y_q(\tilde U_t)}\le \|f-Q_tf\|
_{Y_{p,q,T_{t-1}}(\tilde U_t)}
\stackrel{(\ref{ftp_ti_mod})}{{\underset{\mathfrak{Z}_0}
{\lesssim}}} 2^{-\left(\frac 1p-\frac 1q\right)t}.
\end{align}
Notice that if $t<t_0$, then $Q_tf=Q_{t+1}f=0$ (since $\tilde
P_{t,i_0}=0$ for $t<t_0$).

Throughout we set $\log x:=\log_2 x$.

{\bf Definition of the operators $P_{t,m}$.} For $t\ge t_0$ we set
\begin{align}
\label{mt_def} m_t=\lceil \log \nu_t\rceil.
\end{align}
In \cite{vas_width_raspr} for each $m\in \Z_+$ the set $G_{m,t}
\subset G_t$, the partition \label{tttm}$\tilde T_{t,m}$ of the
set $G_{m,t}$ and the linear continuous operator
\begin{align}
\label{ptm_yq} P_{t,m}:Y_q(\Omega) \rightarrow {\cal S}_{\tilde
T_{t,m}}(\Omega)
\end{align}
we constructed. Here the following properties hold:
\begin{enumerate}
\item $G_{m,t}\subset G_{m+1,t}$, $G_{m_t,t}=G_t$;
\item for any $m\in \Z_+$
\begin{align}
\label{cttm} {\rm card}\, \tilde T_{t,m}
\underset{\mathfrak{Z}_0}{\lesssim} 2^m;
\end{align}
\item for any function $f\in \hat X_p(\Omega)$ and for any set
$E\in \tilde T_{t,m}$
\begin{align}
\label{fpttm1} \|f-P_{t,m}f\|_{Y_q(E)}
\underset{\mathfrak{Z}_0}{\lesssim} 2^{-\left(\frac 1p-\frac
1q\right)t} \|f\|_{X_p(E)}, \quad m\le m_t,
\end{align}
\begin{align}
\label{fpttm2} \|f-P_{t,m}f\|_{Y_q(E)}
\underset{\mathfrak{Z}_0}{\lesssim} 2^{-\left(\frac 1p-\frac
1q\right)t}\cdot 2^{-\delta_*(m-m_t)}\|f\|_{X_p(E)}, \quad m> m_t;
\end{align}
\item for any set $E\in \tilde T_{t,m}$
\begin{align}
\label{card_et} {\rm card}\, \{E'\in \tilde T_{t,m\pm 1}:\ {\rm
mes} (E\cap E')>0\} \underset{\mathfrak{Z}_0}{\lesssim} 1.
\end{align}
\end{enumerate}
Moreover, we may assume that
\begin{align}
\label{tmtt} \tilde T_{m_t,t}=\{\hat F(\xi)\}_{\xi \in {\bf
V}(\Gamma_t)},
\end{align}
\begin{align}
\label{ptmt_def} P_{t,m_t}f|_{\hat F(\xi)}=P_\xi f|_{\hat F(\xi)},
\quad \xi\in {\bf V}(\Gamma_t)
\end{align}
(it follows from the construction in \cite[p.
37--40]{vas_width_raspr}).

\smallskip

Let
\begin{align}
\label{t_st_n} t_*(n)=\min \{t\in \N:\; \overline{\nu}_t\ge n\},
\end{align}
\begin{align}
\label{t_st_st_n} t_{**}(n)=\min \{t\in \N:\; \overline{\nu}_t\ge
2^n\}.
\end{align}
Then
\begin{align}
\label{2tn_est} 2^{t_*(n)} \underset{\mathfrak{Z}_0}{\asymp} \log
n, \quad 2^{t_{**}(n)} \underset{\mathfrak{Z}_0}{\asymp} n
\end{align}
(see \cite[formula (49)]{vas_entr}).

For any $f\in \hat X_p(\Omega)$ the following equality holds:
\begin{align}
\label{f_razl} \begin{array}{c} f=\sum \limits _{t=t_0}^{t_*(n)-1}
(Q_{t+1}f-Q_tf) \chi _{\tilde U_{t+1}} + \\ +\sum \limits
_{t=t_0}^{t_*(n)-1} \sum \limits_{m=0}^\infty
(P_{t,m+1}f-P_{t,m}f)\chi _{G_{m,t}}+ (f-Q_{t_*(n)}f)\chi _{\tilde
U_{t_*(n)}} \end{array}
\end{align}
(it can be proved similarly as formula (82) in
\cite{vas_width_raspr}).

We set
\begin{align}
\label{til_q_def} \tilde Qf|_{G_t}=P_{t,m_t}f|_{G_t}, \quad t\ge
t_0,
\end{align}
\begin{align}
\label{qnm_def}
\begin{array}{c}
\tilde Q_{n,m}f|_{G_t}=P_{t,m_t+m}f|_{G_t} \quad \text{for} \quad
t_*(n)\le t<t_{**}(n), \quad m\in \Z_+,\\ \tilde Q_{n,m}f|_{G_t}=0
\quad \text{for} \quad t<t_*(n)\quad \text{or} \quad t\ge
t_{**}(n).
\end{array}
\end{align}
Since $\|f-P_{t,m_t}f\|_{Y_q(\hat F(\xi))}
\stackrel{(\ref{fpttm1}),
(\ref{tmtt})}{\underset{\mathfrak{Z}_0}{\lesssim}} 2^{-\left(\frac
1p-\frac 1q\right)t}\|f\|_{X_p(\hat F(\xi))}$ for any $\xi \in
{\bf V}(\Gamma_t)$ and the space $Y_q(\Omega)$ is Banach, we get
from the inequality $p<q$ that for any $f\in \hat X_p(\Omega)$ the
inclusion $\tilde Qf \in Y_q(\Omega)$ holds and
\begin{align}
\label{fqf} \|f-\tilde Qf\|_{Y_q(\tilde U_{t_{**}(n)})}
\underset{\mathfrak{Z}_0}{\lesssim} 2^{\left(\frac 1q-\frac
1p\right)t_{**}(n)} \|f\|_{X_p(\Omega)}
\stackrel{(\ref{2tn_est})}{\underset{\mathfrak{Z}_0}{\lesssim}}
n^{\frac 1q-\frac 1p} \|f\|_{X_p(\Omega)}.
\end{align}
Further,
\begin{align}
\label{fmqtsnf} \begin{array}{c} (f-Q_{t_*(n)}f)\chi _{\tilde
U_{t_*(n)}} = (\tilde Qf-Q_{t_*(n)}f)\chi _{\tilde U_{t_*(n)}} +
\\
+ \sum \limits _{m=0}^\infty (\tilde Q_{n,m+1}f-\tilde Q_{n,m}f)
+(f -\tilde Qf) \chi _{\tilde U_{t_{**}(n)}}; \end{array}
\end{align}
indeed, from (\ref{fpttm2}), (\ref{til_q_def}) and (\ref{qnm_def})
it follows that $\sum \limits _{m=0}^\infty (\tilde
Q_{n,m+1}f-\tilde Q_{n,m}f)=(f-\tilde Qf)\chi _{\tilde U_{t_*(n)}
\backslash \tilde U_{t_{**}(n)}}$.

\begin{Lem}
\label{sum_qt_est} There exists a sequence
$\{k_t\}_{t=t_0}^{t_{*}(n)-1}\subset \N$ such that
\begin{align}
\label{slttkt} \sum \limits _{t=t_0}^{t_{*}(n)-1}(k_t-1)
\underset{\mathfrak{Z}_0}{\lesssim} n,
\end{align}
\begin{align}
\label{sum_ekt_2_plq11} \sum \limits _{t=t_0}^{t_{*}(n)-1}
e_{k_t}(Q_{t+1}-Q_t:\hat X_p(\Omega) \rightarrow Y_q(\tilde
U_{t+1}))\underset{\mathfrak{Z}_0}{\lesssim} n^{\frac 1q-\frac
1p}.
\end{align}
\end{Lem}

\begin{Lem}
\label{ptm_sum} There exists a sequence $\{k_{t,m}\}_{t_0\le t<
t_*(n), m\in \Z_+}\subset \N$ such that $\sum \limits
_{t,m}(k_{t,m}-1) \underset{\mathfrak{Z}_0}{\lesssim} n$,
$$
\sum \limits _{t=t_0}^{t_*(n)-1}\sum \limits _{m=0}^\infty
e_{k_{t,m}}(P_{t,m+1}-P_{t,m}:\hat X_p(\Omega) \rightarrow
Y_q(G_{m,t})) \underset{\mathfrak{Z}_0}{\lesssim} n^{\frac
1q-\frac 1p}.
$$
\end{Lem}

Lemmas \ref{sum_qt_est} and \ref{ptm_sum} are proved similarly as
Lemmas 6 and 7 in \cite{vas_entr}.

\begin{Lem}
\label{main_lem} We have
$$
e_n((\tilde Q-Q_{t_*(n)})\chi _{\tilde U_{t_*(n)}}:\hat
X_p(\Omega) \rightarrow
Y_q(\Omega))\underset{\mathfrak{Z}_0}{\lesssim} n^{\frac 1q-\frac
1p}.
$$
\end{Lem}

Notice that for $t_{**}(n)\le t_0$ this estimate follows from
(\ref{fqf}). Hence, throughout we assume that $t_{**}(n)>t_0$.

\smallskip

We need some auxiliary assertions.

For any $\nu \in \N$ we denote by $I_\nu$ the identity operator on
$\R^\nu$.

\begin{trma}
\label{shutt_trm} {\rm \cite{c_schutt, edm_trieb_book}.} Let $1\le
p\le q\le \infty$. Then
$$
e_k(I_\nu:l_p^{\nu}\rightarrow l_q^{\nu}) \underset{p,q}{\asymp}
\left\{ \begin{array}{l} 1, \quad 1\le k\le
\log \nu, \\
\left(\frac{\log\left(1+\frac{\nu}{k}\right)}{k}\right)^{\frac1p
-\frac 1q}, \quad \log \nu\le k\le \nu, \\
2^{-\frac{k}{\nu}}\nu^{\frac 1q-\frac 1p}, \quad \nu\le
k.\end{array}\right.
$$
\end{trma}

The following properties of entropy numbers are well-known (see,
e.g., \cite{edm_trieb_book, piet_op}):
\begin{enumerate}
\item if $T:X\rightarrow Y$, $S:Y\rightarrow Z$ are linear
continuous operators, then $e_{k+l-1}(ST)\le e_k(S)e_l(T)$;
\item if $T, \, S:X\rightarrow Y$ are linear continuous operators,
then
\begin{align}
\label{eklspt} e_{k+l-1}(S+T)\le e_k(S)+e_l(T).
\end{align}
\end{enumerate}
In particular, from the first property it follows that
\begin{align}
\label{mult_n} e_k(ST)\le \|S\|e_k(T), \quad e_k(ST)\le
\|T\|e_k(S).
\end{align}

\begin{trma}
\label{lifs_sta} {\rm \cite{lifs_m}.} Let $X$, $Y$ be normed
spaces, and let $V\in L(X, \, Y)$, $\{V_\nu\}_{\nu\in {\cal
N}}\subset L(X, \, Y)$. Then for any $n\in \N$
$$
e_{n+[\log_2|{\cal N}|]+1}(V)\le \sup _{\nu\in {\cal N}}
e_n(V_\nu)+\sup _{x\in B_X} \inf _{\nu\in {\cal N}} \|Vx-V_\nu
x\|_Y.
$$
\end{trma}

\begin{Lem}
\label{lemma_o_razb_dereva1} {\rm \cite{vas_john}.} Let $({\cal
T}, \, \xi_*)$ be a tree with finite vertex set, let
\begin{align}
\label{cardvvvk} {\rm card}\, {\bf V}_1(\xi)\le k \quad\text{for
any vertex } \; \xi\in {\bf V}({\cal T}),
\end{align}
and let the mapping $\Phi :2^{{\bf V}(\cal T)}\rightarrow \R_+$
satisfy the following conditions:
\begin{align}
\label{prop_psi} \Phi(V_1\cup V_2)\ge \Phi(V_1)+\Phi(V_2), \; V_1,
\, V_2\subset {\bf V}({\cal T}), \;\; V_1\cap V_2=\varnothing,
\end{align}
$\Phi({\bf V}({\cal T}))>0$. Then there is a number $C(k)>0$ such
that for any $n\in \N$ there exists a partition $\mathfrak{S}_n$
of the tree ${\cal T}$ into at most $C(k)n$ subtrees ${\cal T}_j$,
which satisfies the following conditions:
\begin{enumerate}
\item $\Phi({\bf V}({\cal T}_j))\le \frac{(k+2)\Phi({\bf V}({\cal T}))}{n}$
for any $j$ such that ${\rm card}\, {\bf V}({\cal T}_j)\ge 2$;
\item if $m\le 2n$, then each element of
$\mathfrak{S}_n$ intersects with at most $C(k)$ elements of
$\mathfrak{S}_m$.
\end{enumerate}
\end{Lem}

\begin{Lem}
\label{oper_a} {\rm \cite{vas_width_raspr}.} Let $T$ be a finite
partition of a measurable subset $G\subset \Omega$, $\nu=\dim
{\cal S}_T(\Omega)$ (see (\ref{st_omega})). Then there exists a
linear isomorphism $A:{\cal S}_T(\Omega)\rightarrow \R^\nu$ such
that $\|A\|_{Y_{p,q,T}(G)\rightarrow
l_{\sigma_{p,q}}^\nu}\underset{\sigma_{p,q}, \, r_0}{\lesssim} 1$,
$\|A^{-1}\| _{l_q^\nu\rightarrow Y_q(G)} \underset{q, \,
r_0}{\lesssim} 1$.
\end{Lem}

\begin{Lem}
\label{sum_lem} {\rm \cite[formula (60)]{vas_bes}.} Let
$\Lambda_*:(0, \, +\infty) \rightarrow (0, \, +\infty)$ be an
absolutely continuous function such that $\lim \limits_{y\to
+\infty}\frac{y\Lambda _*'(y)} {\Lambda _*(y)}=0$. Then for any
$\varepsilon >0$
\begin{align}
\label{te} t^{-\varepsilon}
\underset{\varepsilon,\Lambda_*}{\lesssim}
\frac{\Lambda_*(ty)}{\Lambda_*(y)}\underset{\varepsilon,
\Lambda_*}{\lesssim} t^\varepsilon,\quad 1\le y<\infty, \;\; 1\le
t<\infty.
\end{align}
\end{Lem}

Given $t_*(n)\le t\le t_{**}(n)$, $i\in \hat J_t$, we denote
$$
\overline{\cal A}_{t,i}= \left\{ \begin{array}{l} {\cal A} _{t,i}
\quad \text{for}\quad t<t_{**}(n), \\ \tilde {\cal
A}_{t_{**}(n),i} \quad \text{for} \quad t=t_{**}(n),
\end{array}\right.
$$
$\overline{\Gamma}_t=\Gamma_t$ for $t<t_{**}(n)$,
$\overline{\Gamma}_{t_{**}(n)}=\tilde \Gamma_{t_{**}(n)}$. Let
${\cal D}$ be a subtree in $\tilde \Gamma_{t_*(n)}$. We set
\begin{align}
\label{hjtd} \hat J_{t,{\cal D}} =\{i\in \hat J_t:\; {\bf
V}(\overline{\cal A}_{t,i}) \cap {\bf V}({\cal D}) \ne
\varnothing\}, \quad \overline{{\cal D}}_{t,i} =\overline{\cal
A}_{t,i} \cap {\cal D}.
\end{align}

%Notice that there is a one-to-one correspondence between the
%family of partitions of a tree $({\cal T}, \, \hat \xi)$ into
%subtrees and the family of subsets ${\bf V}({\cal T})$ that
%contain the vertex $\hat \xi$. Indeed, let $\{({\cal D}_j, \,
%\xi_j)\}_{j\in J}$ be the partition of ${\cal T}$. Then we take
%$\{\xi_j\}_{j\in J}$. Conversely, let $\hat \xi \in
%\{\xi_j\}_{j\in J} \subset {\bf V}({\cal T})$. For each $j\in J$
%we consider the tree ${\cal D}_j$ with vertex set $\{\xi\ge
%\xi_j:\; [\xi_j, \, \xi] \cap \{\xi_i\}_{i\ne j} =\varnothing\}$.
%Then $({\cal D}_j, \, \xi_j)$ is a partition of ${\cal T}$.

Throughout we take as $\varepsilon=\varepsilon(\mathfrak{Z}_0)>0$
a sufficiently small number (it will be chosen later by
$\mathfrak{Z}_0$).

\renewcommand{\proofname}{\bf Proof of Lemma \ref{main_lem}}
\begin{proof} We set $t_*'(n)=\max\{t_*(n), \, t_0\}$.

{\bf Step 1.} Given $t<t_{**}(n)$, we denote by $\hat {\cal A}_t$
the subtree in ${\cal A}$ with vertex set ${\bf V}({\cal A})
\backslash {\bf V}(\tilde \Gamma_{t+1})$.

Let $f\in B\hat X_p(\Omega)$. For each $t'_*(n)\le t< t_{**}(n)$
we define the mapping $\Phi_{f,t}:2^{{\bf V}(\hat {\cal A}_t)}
\rightarrow \R_+$ by
$$
\Phi_{f,t}({\bf W}) =\sum \limits _{\xi \in {\bf W} \cap {\bf
V}({\Gamma} _t)} \|f\|^p_{X_p(\hat F(\xi))}.
$$
Then for any disjoint sets ${\bf W}_1$, ${\bf W}_2$ we have
\begin{align}
\label{phi_ft} \Phi_{f,t}({\bf W}_1 \sqcup {\bf W}_2)
=\Phi_{f,t}({\bf W}_1) +\Phi_{f,t}({\bf W}_2).
\end{align}

For each $t_*'(n)\le t<t_{**}(n)$ we set
\begin{align}
\label{nt} \begin{array}{c} \varepsilon_t = \sum \limits _{\xi \in
{\bf V}(\Gamma_t)}\|f\|^p_{X_p(\hat F(\xi))}; \\ n_t= \lceil
n\cdot 2^{-t} \varepsilon_t\rceil \quad \text{if}\quad
\varepsilon_t>0; \quad n_t=1 \quad \text{if}\quad \varepsilon_t=0.
\end{array}
\end{align}
Then
\begin{align}
\label{sumet} \sum \limits _{t=t_*'(n)}^{t_{**}(n)-1} \varepsilon
_t\le 1,
\end{align}
\begin{align}
\label{sltnt} \sum \limits _{t=t_*'(n)}^{t_{**}(n)-1} 2^t n_t \le
\sum \limits _{t=t_*'(n)}^{t_{**}(n)-1} n \varepsilon_t +\sum
\limits _{t=t_*'(n)}^{t_{**}(n)-1} 2^t \stackrel{(\ref{2tn_est}),
(\ref{sumet})}{\le} \hat C(\mathfrak{Z}_0)n
\end{align}
with $\hat C(\mathfrak{Z}_0)\in \N$. By (\ref{nt}),
\begin{align}
\label{pft_at} \Phi_{f,t}({\bf V}(\hat {\cal A}_t))
=\varepsilon_t.
\end{align}

Let $f\ne 0$. It follows from Lemma \ref{lemma_o_razb_dereva1} and
(\ref{c_v1_a}) that there exists a number $C(\mathfrak{Z}_0)\in
\N$ and a family of partitions $\{{\bf T}_{f,t,l}\}_{0\le l\le
\log n_t}$ of the tree $\hat{\cal A}_t$, which satisfy the
following conditions:
\begin{align}
\label{ctftl} {\rm card}\, {\bf T}_{f,t,l} \le C(\mathfrak{Z}_0)
2^{-l}n_t,
\end{align}
\begin{align}
\label{cabb} {\rm card}\, \{{\cal A}''\in {\bf T}_{f,t,l\pm 1}:\;
{\bf V}({\cal A'})\cap {\bf V}({\cal A}'')\ne \varnothing\}
\underset{\mathfrak{Z}_0}{\lesssim} 1, \quad {\cal A}'\in {\bf
T}_{f,t,l},
\end{align}
and for any subtree ${\cal A}'\in {\bf T}_{f,t,l}$ such that ${\rm
card}\, {\bf V}({\cal A}')\ge 2$
\begin{align}
\label{pft_vapr} \Phi_{f,t}({\bf V}({\cal A}'))
\stackrel{(\ref{pft_at})}{\le} C(\mathfrak{Z}_0) 2^l n_t^{-1}
\varepsilon_t.
\end{align}
Moreover, we may assume that
\begin{align}
\label{tftfl} {\bf T}_{f,t,\lfloor \log n_t \rfloor}=\{\hat{\cal
A}_t\}.
\end{align}
For $f\equiv 0$ we set ${\bf T}_{f,t,l}=\{\hat{\cal A}_t\}$.

Given $t<t_{**}(n)$, we denote
\begin{align}
\label{wt_def} {\bf W}_{t} =\left\{ \xi \in {\bf V}({\Gamma}_t):\;
\{\xi\} \in {\bf T}_{f,t,0}, \; \|f\|_{X_p(\hat F(\xi))}^p >
C(\mathfrak{Z}_0)  n_t^{-1} \varepsilon_t\right\},
\end{align}
\begin{align}
\label{ovr_wt_def} \overline{{\bf W}}_{t}=\{ \xi \in {\bf
W}_{t}:\; {\bf V}_1^{\cal A}(\xi) \cap \{\hat \xi_{t+1,j}\}_{j\in
\hat J_{t+1}} \ne \varnothing\},
\end{align}
\begin{align}
\label{st_def} S_{t} =\{\xi \in {\bf V}({\Gamma}_t):\; \exists
{\cal D}\in {\bf T}_{f,t,0}:\; \xi \in {\bf V}_{\min}({\cal D})\},
\end{align}
\begin{align}
\label{hst_def} \begin{array}{c} \hat S_{t} =S_{t}\cup \left(\cup
_{\xi \in \overline{\bf W}_{t-1}} [{\bf V}_1^{{\cal A}}(\xi)\cap
{\bf V}(\Gamma_t)]\right) \quad \text{for}\quad t<t_{**}(n), \\
\hat S_{t_{**}(n)}=\cup _{\xi \in \overline{\bf W}_{t_{**}(n)-1}}
[{\bf V}_1^{{\cal A}}(\xi) \cap {\bf V}(\Gamma_{t_{**}(n)})],
\end{array}
\end{align}
\begin{align}
\label{hs_def} \hat S=\left( \cup _{t=t_*'(n)}^{t_{**}(n)} \hat
S_{t}\right) \cup \{\xi_{t_*'(n),j}\}_{j\in \hat J_{t_*'(n)}}.
\end{align}
Then
\begin{align}
\label{hsbst} \hat S\backslash \cup _{t=t_*'(n)} ^{t_{**}(n)-1}
S_t \subset \cup_{t=t_*'(n)}^{t_{**}(n)} \{\hat\xi_{t,j}:\; j\in
\hat J_t\}, \qquad {\bf W}_t\subset S_t.
\end{align}

For each vertex $\xi \in \hat S$ we denote by ${\cal D}_{(\xi)}$
the tree with vertex set
\begin{align}
\label{vdxi} {\bf V}({\cal D}_{(\xi)}) =\{\xi'\ge \xi:\; [\xi, \,
\xi']\cap \hat S=\{\xi\}\}.
\end{align}

We set
\begin{align}
\label{bftf} {\bf T}_{f} =\{{\cal D}_{(\xi)}:\; \xi \in \hat S\}.
\end{align}
Then ${\bf T}_{f}$ is a partition of the graph $\tilde
\Gamma_{t_*'(n)}$ into subtrees.

{\bf Step 2.} We say that ${\cal D}\in \tilde {\bf T}_{f}$ if
${\cal D}\in \mathbf{T}_f$ and for any $\xi\in {\bf W}_{t}$ we
have ${\cal D}\ne \{\xi\}$.

Let $({\cal D}, \, \xi_*)\in \tilde {\bf T}_{f}$. Then $\xi_* \in
{\bf V}(\tilde\Gamma_{t'_*(n)})$. From Assumption \ref{sup1} and
the inequality $p<q$ it follows that
$$
\|f-P_{\xi_*}f\|^q_{Y_q(\Omega_{\cal D})}
\stackrel{(\ref{f_pom_f}),
(\ref{hjtd})}{\underset{\mathfrak{Z}_0}{\lesssim}} \sum \limits
_{t=t_*'(n)} ^{t_{**}(n)} 2^{t\left(1-\frac qp\right)}\sum \limits
_{j\in \hat J_{t,{\cal D}}} \|f\|^q _{X_p(\Omega_{\overline{{\cal
D}}_{t,j}})},
$$
$$
\|f-\tilde Qf\|^q_{Y_q(\Omega_{\cal D})}
\stackrel{(\ref{til_q_def})}{=} \sum \limits _{t=t_*'(n)}
^{t_{**}(n)} \sum \limits _{j\in \hat J_{t,{\cal D}}}
\|f-P_{t,m_t}f\|^q_{Y_q(\Omega_{\overline{\cal D}_{t,j}})}
\stackrel{(\ref{fpttm1})}{\underset{\mathfrak{Z}_0}{\lesssim}}
$$
$$
\lesssim \sum \limits _{t=t_*'(n)} ^{t_{**}(n)} 2^{t\left(1-\frac
qp\right)}\sum \limits _{j\in \hat J_{t,{\cal D}}} \|f\|^q
_{X_p(\Omega_{\overline{{\cal D}}_{t,j}})}.
$$
Hence,
\begin{align}
\label{m1} \sum \limits _{({\cal D}, \, \xi_*)\in \tilde {\bf
T}_{f}} \|f-P_{\xi_*}f\|^q_{Y_q(\Omega_{\cal D})}
\underset{\mathfrak{Z}_0}{\lesssim} \sum \limits _{t=t_*'(n)}
^{t_{**}(n)} 2^{t\left(1-\frac qp\right)}\sum \limits _{({\cal D},
\, \xi_*)\in \tilde {\bf T}_{f}}\sum \limits _{j\in \hat
J_{t,{\cal D}}} \|f\|^q _{X_p(\Omega_{\overline{{\cal
D}}_{t,j}})},
\end{align}
\begin{align}
\label{m2} \sum \limits _{({\cal D}, \, \xi_*)\in \tilde {\bf
T}_{f}} \|f-\tilde Qf\|^q_{Y_q(\Omega_{\cal D})}
\underset{\mathfrak{Z}_0}{\lesssim} \sum \limits _{t=t_*'(n)}
^{t_{**}(n)} 2^{t\left(1-\frac qp\right)}\sum \limits _{({\cal D},
\, \xi_*)\in \tilde {\bf T}_{f}}\sum \limits _{j\in \hat
J_{t,{\cal D}}} \|f\|^q _{X_p(\Omega_{\overline{{\cal
D}}_{t,j}})}.
\end{align}

Denote by $\tilde{\bf T}_{f,t}$ the family of trees $\tilde {\cal
D} \in {\bf T}_{f,t,0}$ such that
\begin{align}
\label{ttft} \sum \limits _{\xi \in {\bf V}(\tilde {\cal D}) \cap
{\bf V}(\Gamma_t)} \|f\|^p _{X_p(\hat F(\xi))} \le
C(\mathfrak{Z}_0)n_t^{-1}\varepsilon_t.
\end{align}

{\bf Assertion 1.} Let ${\cal D}\in \tilde{\bf T}_{f}$,
$t_*'(n)\le t<t_{**}(n)$, $j\in J_{t,{\cal D}}$. Then there exists
a tree ${\cal T} \in \tilde {\bf T}_{f,t}$ such that ${\cal
D}_{t,j}={\cal T}_{t,j}$.

{\it Proof of Assertion 1.} Let ${\cal D}={\cal D}_{(\xi_*)}$,
$\xi_*\in \hat S$. Since $j\in J_{t,{\cal D}}$, the vertices
$\xi_*$ and $\hat \xi_{t,j}$ are comparable. There exists a tree
${\cal T} \in {\bf T}_{f,t,0}$ such that
\begin{align}
\label{mxi} \max\{\xi_*,\hat \xi_{t,j}\} \in {\bf V}({\cal T}).
\end{align}
We claim that ${\bf V}({\cal D}_{t,j}) \subset {\bf V}({\cal T})$.
Indeed, let $\xi\in {\bf V}({\cal D}_{t,j}) \backslash {\bf
V}({\cal T})$. Then $\xi> \max\{\hat \xi_{t,j}, \, \xi_*\}$. We
set
$$
\eta_*=\min \{\eta \in [\max\{\xi_*, \, \hat \xi_{t,j}\}, \,
\xi]:\; \eta \notin {\bf V}({\cal T})\} >\xi_*.
$$
Then $\eta_*\in {\bf V}(\Gamma_t)$ and there exists a tree ${\cal
T}'\in {\bf T}_{f,t,0}$ such that $\eta_*$ is the minimal vertex
of ${\cal T}'$. Hence, $\eta \stackrel{(\ref{st_def})}{\in} S_t
\stackrel{(\ref{hst_def}),(\ref{hs_def})}{\subset} \hat S$;
therefore, $\xi \stackrel{(\ref{vdxi})}{\notin} {\bf V}({\cal
D})$, which leads to a contradiction.

Thus, ${\bf V}({\cal T}_{t,j})\supset {\bf V}({\cal D}_{t,j}) \ne
\varnothing$.

Let us show that ${\bf V}({\cal T}_{t,j})\subset {\bf V}({\cal
D}_{t,j})$. Denote by $\hat \xi$ the minimal vertex of the tree
${\cal T}$. Since ${\bf V}({\cal T}_{t,j})\ne \varnothing$, the
vertices $\hat \xi$ and $\hat \xi_{t,j}$ are comparable. Let us
show that $\max\{\xi_*, \, \hat \xi_{t,j}\}=\max\{\hat \xi, \,
\hat \xi_{t,j}\}$. Indeed, by (\ref{mxi}) we have $\max\{\xi_*, \,
\hat \xi_{t,j}\}\ge \max\{\hat \xi, \, \hat \xi_{t,j}\}$. If the
inequality is strict, then $\xi_*>\hat \xi_{t,j}$. In addition,
$\xi_*\in {\bf V}({\cal A}_{t,j}) \cap \hat S$, and by
(\ref{st_def}) and (\ref{hsbst}) we get $\xi_*\in S_t$; i.e.,
$\xi_*$ is the minimal vertex of some tree from the partition
${\bf T}_{f,t,0}$. By (\ref{mxi}), $\xi_*{=}\hat \xi$.

Thus,
\begin{align}
\label{max_xi} \max\{\hat \xi, \, \hat \xi_{t,j}\}=\max\{\xi_*, \,
\hat \xi_{t,j}\}\in {\bf V}({\cal D}_{t,j}).
\end{align}
Suppose that there exists a vertex $\xi\in {\bf V}({\cal T}_{t,j})
\backslash {\bf V}({\cal D})$. We set
\begin{align}
\label{eta_st} \eta_{**}=\min \left(\left[\max\{\hat \xi, \, \hat
\xi_{t,j}\}, \, \xi \right]\backslash {\bf V}({\cal D})\right).
\end{align}
Then $\eta_{**}>\max\{\hat \xi, \, \hat \xi_{t,j}\}$ and
$\eta_{**}\in {\bf V}({\cal A}_{t,j})$. Hence, $\eta_{**} \notin
\cup _{t'=t_*'(n)}^{t_{**}(n)} \{\hat \xi _{t',i}:\; i\in \hat
J_{t'}\}$. Further, $\eta_{**}\in {\bf V}({\cal T}_{t,j})$;
therefore, $\eta_{**}\notin \cup _{t'=t_*'(n)} ^{t_{**}(n)}
S_{t'}$ by (\ref{st_def}). From (\ref{hsbst}) it follows that
$\eta_{**}\notin \hat S$. This together with (\ref{vdxi}),
(\ref{max_xi}), (\ref{eta_st}) yields that $[\xi_*, \,
\eta_{**}]\cap \hat S=\{\xi_*\}$; applying (\ref{vdxi}) once
again, we get that $\eta_{**}\in {\bf V}({\cal D})$, which leads
to a contradiction.

It remains to check that ${\cal T}\in \tilde {\bf T}_{f,t}$. If
${\cal T}\notin \tilde {\bf T}_{f,t}$, then ${\bf V}({\cal T})
\stackrel{(\ref{pft_vapr}), (\ref{ttft})}{=} \{\xi\}$,
$\|f\|_{X_p(\hat F(\xi))}
>C(\mathfrak{Z}_0) n_t^{-1} \varepsilon_t$; i.e., $\xi
\stackrel{(\ref{wt_def})}{\in} {\bf W}_t$. Then either $\xi \in
\overline{\bf W}_t$ (in this case, ${\bf V}_1^{{\cal A}}(\xi)
\stackrel{(\ref{st_def}), (\ref{hst_def})}{\subset} S_t \cup \hat
S_{t+1}$) or ${\bf V}_1^{{\cal A}}(\xi)
\stackrel{(\ref{st_def})}{\subset} S_t$. Hence, ${\bf V}_1^{{\cal
A}}(\xi) \stackrel{(\ref{hs_def})}{\subset} \hat S$. This implies
that ${\bf V}({\cal D}) \stackrel{(\ref{vdxi})}{=}\{\xi\}$ and
${\cal D} \notin \tilde{\bf T}_f$. This completes the proof of
Assertion 1. $\diamond$

By (\ref{st_def}), the first inclusion of (\ref{hsbst}),
(\ref{vdxi}) and (\ref{bftf}), for any tree ${\cal D}\in {\bf
T}_f$ we have $\overline{{\cal D}}_{t_{**}(n),j}=\tilde {\cal
A}_{t_{**}(n),j}$, $j\in J_{t_{**}(n)}$. This together with
Assertion 1 yields that
$$
\sum \limits _{t=t_*'(n)} ^{t_{**}(n)} 2^{t\left(1-\frac
qp\right)}\sum \limits _{({\cal D}, \, \xi_*)\in \tilde {\bf
T}_{f}}\sum \limits _{j\in \hat J_{t,{\cal D}}} \|f\|^q
_{X_p(\Omega_{\overline{{\cal D}}_{t,j}})} \le \sum \limits
_{t=t_*'(n)}^{t_{**}(n)-1}2^{t\left(1-\frac qp\right)} \sum
\limits _{{\cal T} \in \tilde {\bf T}_{f,t}}\sum \limits _{j\in
J_{t, {\cal T}}} \|f\|^q_{X_p(\Omega _{{\cal T}_{t,j}})}+
$$
$$
+ 2^{t_{**}(n)\left(1-\frac qp\right)}\sum \limits _{j\in \hat
J_{t_{**}(n)}} \|f\|^q_{X_p(\Omega _{\tilde{\cal
A}_{{t_{**}(n),j}}})}
\stackrel{(\ref{2tn_est})}{\underset{\mathfrak{Z}_0}{\lesssim}}
n^{1-\frac qp}+
$$
$$
+\sum \limits _{t=t_*'(n)}^{t_{**}(n)-1}2^{t\left(1-\frac
qp\right)} \sum \limits _{{\cal T} \in \tilde {\bf T}_{f,t}}
\left(\sum \limits _{\xi \in {\bf V}({\cal T})\cap {\bf
V}({\Gamma_t})} \|f\|^p_{X_p(\hat F(\xi))}\right)^{\frac
qp}\stackrel{(\ref{ttft})}{\underset{\mathfrak{Z}_0}{\lesssim}}
$$
$$
\lesssim n^{1-\frac qp}+\sum \limits
_{t=t_*'(n)}^{t_{**}(n)-1}2^{t\left(1-\frac qp\right)} \sum
\limits _{{\cal T} \in \tilde {\bf T}_{f,t}} \varepsilon_t^{\frac
qp}n_t^{-\frac qp}
\stackrel{(\ref{ctftl})}{\underset{\mathfrak{Z}_0}{\lesssim}}
n^{1-\frac qp}+\sum \limits
_{t=t_*'(n)}^{t_{**}(n)-1}2^{t\left(1-\frac qp\right)}
\varepsilon_t^{\frac qp} n_t^{1-\frac qp};
$$
i.e.,
\begin{align}
\label{m3} \sum \limits _{t=t_*'(n)} ^{t_{**}(n)}
2^{t\left(1-\frac qp\right)}\sum \limits _{({\cal D}, \, \xi_*)\in
\tilde {\bf T}_{f}}\sum \limits _{j\in \hat J_{t,{\cal D}}}
\|f\|^q _{X_p(\Omega_{\overline{{\cal D}}_{t,j}})}
{\underset{\mathfrak{Z}_0}{\lesssim}} n^{1-\frac qp}+\sum \limits
_{t=t_*'(n)}^{t_{**}(n)-1}2^{t\left(1-\frac qp\right)}
\varepsilon_t^{\frac qp} n_t^{1-\frac qp}=:A.
\end{align}
If $n\cdot 2^{-t}\varepsilon_t\ge 1$, then
$n_t\stackrel{(\ref{nt})}{\asymp} n\cdot 2^{-t}\varepsilon_t$;
hence,
$$
2^{t\left(1-\frac qp\right)} \varepsilon_t^{\frac qp} n_t^{1-\frac
qp} \underset{p,q}{\asymp} 2^{t\left(1-\frac qp\right)}
\varepsilon_t^{\frac qp} n^{1-\frac qp} 2^{-t\left(1-\frac
qp\right)} \varepsilon_t^{1-\frac qp}=  n^{1-\frac qp}
\varepsilon_t.
$$
If $n\cdot 2^{-t}\varepsilon_t< 1$, then $n_t=1$. From
(\ref{sumet}) it follows that $\varepsilon_t\le 1$. Therefore,
\begin{align}
\label{m4} A\underset{\mathfrak{Z}_0}{\lesssim} n^{1-\frac
qp}+\sum \limits _{t=t_*'(n)}^{t_{**}(n)-1}  n^{1-\frac
qp}\varepsilon_t +\sum \limits _{t=t_*'(n)}^{t_{**}(n)-1}
2^{t\left(1-\frac qp\right)}
\stackrel{(\ref{sumet})}{\underset{\mathfrak{Z}_0}{\lesssim}}
n^{1-\frac qp}+ 2^{t_{**}(n)\left(1-\frac qp\right)}
\stackrel{(\ref{2tn_est})}{\underset{\mathfrak{Z}_0}{\lesssim}}
n^{1-\frac qp}.
\end{align}

From (\ref{m1}), (\ref{m2}), (\ref{m3}), (\ref{m4}) we get that
$$
\sum \limits _{({\cal D}, \, \xi_*)\in \tilde {\bf T}_{f}}
(\|f-P_{\xi_*}f\|^q_{Y_q(\Omega_{\cal D})}+\|f-\tilde
Qf\|^q_{Y_q(\Omega_{\cal D})}) \underset{\mathfrak{Z}_0}{\lesssim}
n^{1-\frac qp};
$$
this yields
\begin{align}
\label{sum_f_pf_yq} \sum \limits _{({\cal D}, \, \xi_*)\in \tilde
{\bf T}_{f}} \|\tilde Qf-P_{\xi_*}f\|^q_{Y_q(\Omega_{\cal D})}
\underset{\mathfrak{Z}_0}{\lesssim}  n^{1-\frac qp}.
\end{align}

{\bf Step 3.} Let ${\bf T}$ be a partition of $\tilde
\Gamma_{t_*'(n)}$ into subtrees. For each $({\cal D}, \, \xi_*)
\in {\bf T}$ we set
\begin{align}
\label{ptf} P_{\bf T}f|_{\Omega_{{\cal D}}}
=P_{\xi_*}f-Q_{t_*'(n)}f.
\end{align}
In addition, we put
\begin{align}
\label{ptf_dop} P_{\bf T}f|_{\Omega\backslash \tilde
U_{t_*'(n)}}=0.
\end{align}
Then
$$
\|\tilde Qf-Q_{t_*'(n)}f-P_{\mathbf{T}}f\|^q_{Y_q(\tilde
U_{t_*'(n)})} =\sum \limits _{({\cal D}, \, \xi_*)\in {\bf T}}
\|\tilde Qf -P_{\xi_*}f\|^q_{Y_q(\Omega_{\cal D})}
\stackrel{(\ref{ptmt_def}),(\ref{til_q_def})}{=}
$$
$$
=\sum \limits _{({\cal D}, \, \xi_*)\in {\bf T}, \, {\rm card}\,
{\bf V}({\cal D})\ge 2}\|\tilde Qf
-P_{\xi_*}f\|^q_{Y_q(\Omega_{\cal D})}.
$$
In particular, from (\ref{sum_f_pf_yq}) and the definition of
$\tilde {\bf T}_f$ (see the beginning of Step 2) it follows that
\begin{align}
\label{fpf_appr} \|\tilde Qf-Q_{t_*'(n)}f
-P_{\mathbf{T}_f}f\|^q_{Y_q(\tilde U_{t_*'(n)})}
\underset{\mathfrak{Z}_0}{\lesssim} n^{1-\frac qp}.
\end{align}

{\bf Step 4.} Let us define the family ${\cal N}$ of partitions of
the graph $\tilde \Gamma_{t_*'(n)}$ into subtrees. Each of these
partitions is constructed as follows. Let $C(\mathfrak{Z}_0)$,
$\hat C(\mathfrak{Z}_0)$ be as defined at Step 1.
\begin{enumerate}
\item We choose the sequence
$\{n_t\}_{t=t_*'(n)}^{t_{**}(n)-1}\subset \N$ such that
\begin{align}
\label{sl_tnt} \sum \limits _{t=t_*'(n)}^{t_{**}(n)-1} 2^t n_t\le
\hat C(\mathfrak{Z}_0)n.
\end{align}
\item For each $t\in \{t_*'(n), \, \dots, \, t_{**}(n)-1\}$
we take $C(\mathfrak{Z}_0)n_t$ vertices in ${\bf V}(\Gamma_t)$
(some of them may coincide); we denote this set by ${\bf U}_t$.
\item For each $t\in \{t_*'(n), \, \dots, \, t_{**}(n)-1\}$ we
choose an arbitrary subset $\hat{\bf V}_t$ of $\{ \xi \in {\bf
U}_t:\; {\bf V}_1^{{\cal A}}(\xi) \cap {\bf V}(\Gamma_{t+1}) \ne
\varnothing\}$.
\item Let
$$
{\bf U}=\cup _{j\in \hat J_{t_*'(n)}} \{\hat \xi_{t_*'(n),j}\}
\cup \left(\cup _{t=t_*'(n)}^{t_{**}(n)-1} {\bf U}_t\right) \cup
\left(\cup _{t=t_*'(n)}^{t_{**}(n)-1} \cup_{\xi\in \hat{\bf V}_t}
[{\bf V}_1^{{\cal A}}(\xi) \cap {\bf V}(\Gamma_{t+1})]\right).
$$
This vertex set generates the desired partition ${\bf T}$ of the
graph $\tilde \Gamma_{t_*'(n)}$ into subtrees:
\begin{align}
\label{tudef} {\bf T}=\{{\cal D}'_{(\xi)}:\; \xi\in {\bf U}\},
\quad {\bf V}({\cal D}'_{(\xi)}) =\{\xi'\ge \xi:\; [\xi, \,
\xi']\cap {\bf U}=\{\xi\}\}.
\end{align}
\end{enumerate}
Let us estimate the value $|{\cal N}|$.

\begin{enumerate}
\item First we estimate the number of choices of
$\{n_t\}_{t=t_*'(n)}^{t_{**}(n)-1}$ (we denote this value by
$N_1$). Let $1\le l\le \hat C({\mathfrak{Z}_0})n$. The number of
choices of $n_t\in \N$ such that $\sum \limits
_{t=t_*'(n)}^{t_{**}(n)-1} 2^t n_t=l$ can be estimated from above
by the number of choices of numbers $\hat n_t\in \N$ such that
$\sum \limits _{t=t_*'(n)}^{t_{**}(n)-1}\hat n_t=l$. The last
magnitude can be estimated from above by the number of partitions
of $\{1, \, \dots, \, l\}$ into $t_{**}(n)-t_*'(n)$ intervals.
This value does not exceed $(l+ t_{**}(n)-t_*'(n)) ^{t_{**}(n)-
t_*'(n)-1}$. Hence,
$$
N_1\le \sum \limits _{1\le l\le \hat C({\mathfrak{Z}_0})n} (l+
t_{**}(n)-t_*'(n))^{t_{**}(n)-t_*'(n)-1}\le $$$$\le\sum \limits
_{k=1}^{\hat C({\mathfrak{Z}_0})n+t_{**}(n)} k^{t_{**}(n)-1}
\underset{\mathfrak{Z}_0}{\lesssim} (\hat
C({\mathfrak{Z}_0})n+t_{**}(n))^{t_{**}(n)}=:N_1'.
$$
\item Given the sequence $\{n_t\}_{t=t_*'(n)} ^{t_{**}(n)-1}$, we
estimate the number of choices of a set $\cup
_{t=t_*'(n)}^{t_{**}(n)-1} {\bf U}_t$ (we denote this magnitude by
$N_2$). We have
$$
N_2\stackrel{(\ref{nu_t_k1})}{\le} \prod
_{t=t_*'(n)}^{t_{**}(n)-1}
(c_3\overline{\nu}_t)^{C(\mathfrak{Z}_0)n_t}=:N_2'.
$$
\item Let us estimate the number of choices of $\hat{\bf V}_t$, $t_*'(n)
\le t\le t_{**}(n)-1$ (denote this value by $N_3$). Since ${\rm
card}\, {\bf U}_t\le C(\mathfrak{Z}_0)n_t$, we have
$$
N_3\le \prod _{t=t_*'(n)}^{t_{**}(n)-1} 2^{C(\mathfrak{Z}_0)n_t}
=:N_3'.
$$
\end{enumerate}
Thus, $|{\cal N}|\le N_1'N_2'N_3'$, which yields
$$
\log |{\cal N}| \le t_{**}(n) \log (\hat C({\mathfrak{Z}_0})n+
t_{**}(n)) +\sum \limits _{t=t_*'(n)}^{t_{**}(n)-1}
C(\mathfrak{Z}_0)n_t \log (c_3\overline{\nu}_t) +$$$$+\sum \limits
_{t=t_*'(n)}^{t_{**}(n)-1} C(\mathfrak{Z}_0)n_t
\stackrel{(\ref{nu_t_k1}),(\ref{2tn_est}),
(\ref{te})}{\underset{\mathfrak{Z}_0}{\lesssim}}
 (\log n)^2+\sum \limits _{t=t_*'(n)}^{t_{**}(n)-1} 2^tn_t
\stackrel{(\ref{sl_tnt})}{\underset{\mathfrak{Z}_0}{\lesssim}} n;
$$
i.e.,
\begin{align}
\label{logn} \log |{\cal N}|{\underset{\mathfrak{Z}_0}{\lesssim}}
n.
\end{align}

{\bf Step 5.} Denote by ${\cal N}'$ the family of partitions ${\bf
T}_f$, $f\in B\hat X_p(\Omega)$ (see (\ref{bftf})). Then ${\cal
N}' \subset {\cal N}$ (it follows from
(\ref{wt_def})--(\ref{hs_def}), (\ref{vdxi}), (\ref{tudef}) and
from the estimates (\ref{sltnt}), (\ref{ctftl})). By (\ref{logn})
and Theorem \ref{lifs_sta}, there exists
$l_*=l_*(\mathfrak{Z}_0)\in \N$ such that
$$
e_{l_*n}(\tilde Q-Q_{t_*'(n)}:\hat X_p(\Omega) \rightarrow
Y_q(\tilde U_{t_*'(n)})) \le
$$
$$
\le\sup _{f\in B\hat X_p(\Omega)} \inf _{{\bf T}\in {\cal N}'}
\|\tilde Qf-Q_{t_*'(n)}f-P_{{\bf T}}f\| _{Y_q(\tilde U_{t_*'(n)})}
+\sup _{{\bf T}\in {\cal N}'} e_n(P_{{\bf T}}:\hat X_p(\Omega)
\rightarrow Y_q(\tilde U_{t_*'(n)})) \le
$$
$$
\le \sup _{f\in B\hat X_p(\Omega)} \|\tilde Qf-Q_{t_*'(n)}f-
P_{\mathbf{T}_f}f\|_{Y_q(\tilde U_{t_*'(n)})}+\sup _{{\bf T}\in
{\cal N}'} e_n(P_{{\bf T}}:\hat X_p(\Omega) \rightarrow Y_q(\tilde
U_{t_*'(n)})) \underset {\mathfrak{Z}_0} {\lesssim}
$$
$$
\stackrel{(\ref{fpf_appr})}{\underset {\mathfrak{Z}_0} {\lesssim}}
n^{\frac 1q-\frac 1p}+\sup _{{\bf T}\in {\cal N}'} e_n(P_{{\bf
T}}:\hat X_p(\Omega) \rightarrow Y_q(\tilde U_{t_*'(n)})).
$$

{\bf Step 6.} It remains to prove that
\begin{align}
\label{enptf} e_n(P_{{\bf T}_f}:\hat X_p(\Omega) \rightarrow
Y_q(\tilde U_{t_*'(n)}))\underset{\mathfrak{Z}_0}{\lesssim}
n^{\frac 1q-\frac 1p},
\end{align}
where $f\in B\hat X_p(\Omega)$.

Recall that ${\bf T}_f=\{{\cal D}_{(\xi)}:\; \xi \in \hat S\}$,
where $\hat S$ is defined by (\ref{hs_def}) and ${\cal D}_{(\xi)}$
is defined by formula (\ref{vdxi}). Also we observe that if
$\xi=\hat\xi_{t_*'(n),j}$ for some $j\in \hat J_{t_*'(n)}$, then
\begin{align}
\label{ptf0} P_{{\bf T}_f}|_{\Omega_{{\cal D}_{(\xi)}}}
\stackrel{(\ref{ptf})}{=} (P_{\hat
\xi_{t_*'(n),j}}-Q_{t_*'(n)})|_{\Omega_{{\cal D}_{(\hat
\xi_{t_*'(n),j})}}} \stackrel{(\ref{qtf_x})}{=} 0.
\end{align}

Let $t_*'(n)\le t<t_{**}(n)$. Recall that the sets $S_t$ and
$\overline{{\bf W}}_t$ are defined by formulas (\ref{st_def}) and
(\ref{ovr_wt_def}), respectively. We set
\begin{align}
\label{vft_def} {\bf V}_{f,t}=S_t \cup \{\xi \in {\bf
V}(\Gamma_{t+1}) \backslash S_{t+1}:\; \exists \eta\in
\overline{\bf W}_t:\; \xi \in {\bf V}_1^{\cal A}(\eta)\}.
\end{align}
Notice that ${\bf V}_{f,t} \subset \hat S$ by (\ref{hst_def}),
(\ref{hs_def}). Denote by $\hat \Gamma_{f,t}$ the maximal subgraph
in ${\cal A}$ on vertex set
\begin{align}
\label{vgft} {\bf V}(\hat \Gamma_{f,t})=\cup _{\xi \in {\bf
V}_{f,t}} {\bf V}({\cal D}_{(\xi)}).
\end{align}
We set
\begin{align}
\label{pft} \Omega_{(f,t)} =\cup _{\xi \in {\bf V}(\hat
\Gamma_{f,t})} \hat F(\xi), \quad P_{(f,t)}h=P_{{\bf T}_f}h
\cdot\chi _{\Omega_{(f,t)}}, \quad h\in Y_q(\Omega).
\end{align}
Denote
\begin{align}
\label{tpr_ft} {\bf T}'_{f,t}=\{{\cal D}_{(\xi)}\}_{\xi \in {\bf
V}_{f,t}}.
\end{align}
Then ${\bf T}'_{f,t}$ is a partition of the graph $\hat \Gamma
_{f,t}$.

We claim that
\begin{align}
\label{sum_ptft} P_{{\bf T}_f} =\sum \limits _{t=t_*'(n)}
^{t_{**}(n)-1} P_{(f,t)}.
\end{align}
Indeed, let $h\in Y_q(\Omega)$. If $x\in \cup
_{t=t'_*(n)}^{t_{**}(n)-1} \cup _{\xi \in {\bf V}_{f,t}} \Omega
_{{\cal D}_{(\xi)}}$, then by (\ref{pft}) we get $P_{{\bf
T}_f}h(x)=\sum \limits _{t=t_*'(n)} ^{t_{**}(n)-1} P_{(f,t)}
h(x)$. In other cases we have $\sum \limits _{t=t_*'(n)}
^{t_{**}(n)-1} P_{(f,t)} h(x)=0$. On the other hand, if $P_{{\bf
T}_f}h(x)\ne 0$, then $x\in \Omega _{{\cal D}_{(\xi)}}$ for some
$\xi \in \hat S$ by (\ref{bftf}) and (\ref{ptf_dop}). From
(\ref{hst_def}), (\ref{hs_def}), (\ref{ptf0}) and (\ref{vft_def})
it follows that $\xi \in {\bf V}_{f,t}$ for some $t\in \{t'_*(n),
\, \dots, \, t_{**}(n)-1\}$. This completes the proof of
(\ref{sum_ptft}).

Let us prove that there exists a sequence $\{k_t\}_{t_*'(n)\le
t<t_{**}(n)}$ such that
$$\sum \limits _{t=t_*'(n)}^{t_{**}(n)-1} (k_t-1)
\underset{\mathfrak{Z}_0}{\lesssim} n,$$
\begin{align}
\label{slttn} \sum \limits _{t=t_*'(n)}^{t_{**}(n)-1}
e_{k_t}(P_{(f,t)} : \hat X_p(\Omega) \rightarrow
Y_q(\Omega))\underset {\mathfrak{Z}_0}{\lesssim} n^{\frac 1q-\frac
1p}.
\end{align}

Let $\xi \in S_t$. If $\xi \in S_t \backslash \overline{\bf W}_t$,
then we set
\begin{align}
\label{tilpft} \tilde {\cal D}_{(\xi)}={\cal D}_{(\xi)}, \quad
\tilde P_{(f,t)}|_{\Omega_{\tilde {\cal D}_{(\xi)}}}
=P_{(f,t)}|_{\Omega_{{\cal D}_{(\xi)}}}
\stackrel{(\ref{ptf}),(\ref{pft})}{=} (P_\xi -
Q_{t_*'(n)})|_{\Omega_{{\cal D}_{(\xi)}}}.
\end{align}
If $\eta \in \overline{\bf W}_t$, then we denote by $\tilde {\cal
D}_{(\eta)}$ the tree with vertex set
\begin{align}
\label{vtilde} {\bf V}(\tilde {\cal D}_{(\eta)})=\{\eta\}\cup
\left(\cup _{\xi \in {\bf V}_1^{\cal A}(\eta)\cap {\bf
V}(\Gamma_{t+1}) \backslash S_{t+1}} {\bf V}({\cal
D}_{(\xi)})\right)
\end{align}
and put
\begin{align}
\label{tilpt} \tilde P_{(f,t)}|_{\Omega_{\tilde{\cal D}_{(\eta)}}}
=(P_\eta-Q_{t_*'(n)})|_{\Omega_{\tilde{\cal D}_{(\eta)}}}.
\end{align}
Observe that by (\ref{vgft}) we have
\begin{align}
\label{tprpr_ft} {\bf T}''_{f,t}:= \{\tilde{\cal D}_{(\xi)}\}_{\xi
\in S_t} \quad\text{is the partition of }\hat \Gamma _{f,t}.
\end{align}
Finally, we set
\begin{align}
\label{dop} \tilde P_{(f,t)}|_{\Omega \backslash \Omega_{(f,t)}}
=0.
\end{align}

Let $$T'_{f,t}=\{\Omega_{{\cal D}}:\; {\cal D}\in {\bf
T}'_{f,t}\}, \quad T''_{f,t}=\{\Omega_{{\cal D}}:\; {\cal D}\in
{\bf T}''_{f,t}\}.$$ By (\ref{ptf}), (\ref{pft}), (\ref{tilpft}),
(\ref{tilpt}), (\ref{dop}) we have
\begin{align}
\label{p_ft_s} P_{(f,t)} \in {\cal S}_{{ T}'_{f,t}}(\Omega), \quad
\tilde P_{(f,t)} \in {\cal S}_{{ T}''_{f,t}}(\Omega).
\end{align}
Moreover, the partition $T'_{f,t}$ refines the partition
$T''_{f,t}$.

Let $h\in B\hat X_p(\Omega)$. Then
$$
\|P_{(f,t)}h-\tilde P_{(f,t)}h\|_{p,q, T'_{f,t}} ^p
\stackrel{(\ref{fpqt}),(\ref{ptf}), (\ref{tpr_ft}),
(\ref{tilpft}), (\ref{vtilde}), (\ref{tilpt})}{=}
$$
$$
= \sum
\limits _{\eta \in \overline{\bf W}_t} \sum \limits _{\xi \in {\bf
V}_1^{\cal A}(\eta) \cap {\bf V}(\Gamma_{t+1})\backslash S_{t+1}}
\| P_\xi h-P_\eta h\|^p_{Y_q(\Omega_{{\cal D}_{(\xi)}})}
\underset{\mathfrak{Z}_0}{\lesssim}
$$
$$
\lesssim \sum \limits _{\eta \in \overline{\bf W}_t} \sum \limits
_{\xi \in {\bf V}_1^{\cal A}(\eta) \cap {\bf V}(\Gamma_{t+1})
\backslash S_{t+1}} (\|h-P_\xi h\|^p_{Y_q(\Omega_{{\cal
D}_{(\xi)}})} +\|h-P_\eta h\|^p_{Y_q(\Omega_{{\cal D}_{(\xi)}})})
\stackrel{(\ref{f_pom_f_cor})}{\underset{\mathfrak{Z}_0}{\lesssim}}
2^{-t \left(1-\frac pq\right)};
$$
i.e.,
\begin{align}
\label{pfth} \|P_{(f,t)}h-\tilde P_{(f,t)}h\|_{p,q, T'_{f,t}}
\underset{\mathfrak{Z}_0}{\lesssim} 2^{-t \left(\frac 1p-\frac
1q\right)}.
\end{align}

We set $k'_t =\lceil n\cdot 2^{-\varepsilon(t-t_*(n))}\rceil$,
where $\varepsilon>0$ is a sufficiently small number (it is chosen
by $\mathfrak{Z}_0$). For any $t\le t_{**}(n)$ we have
$$
{\rm card}\, {\bf T}'_{f,t} \stackrel{(\ref{tpr_ft})}{=} {\rm
card}\, {\bf V}_{f,t} \stackrel{(\ref{c_v1_a}), (\ref{wt_def}),
(\ref{ovr_wt_def}),(\ref{st_def}), (\ref{vft_def})}{\underset
{\mathfrak{Z}_0}{\lesssim}} {\rm card}\, S_t
\stackrel{(\ref{st_def})}{\underset{\mathfrak{Z}_0}{\lesssim}}
{\rm card}\, {\bf T}_{f,t,0} \stackrel {(\ref{nt}),(\ref{sumet}),
(\ref{ctftl})}{\underset{\mathfrak{Z}_0}{\lesssim}} \lceil 2^{-t}n
\rceil
\stackrel{(\ref{2tn_est})}{\underset{\mathfrak{Z}_0}{\asymp}}
2^{-t}n.
$$
This together with (\ref{mult_n}), (\ref{p_ft_s}), (\ref{pfth}),
Theorem \ref{shutt_trm} and Lemma \ref{oper_a} implies that for
some $c_*=c_*(\mathfrak{Z}_0)>0$
\begin{align}
\label{sum_e_p}
\begin{array}{c}
\sum \limits _{t=t_*'(n)}^{t_{**}(n)-1} e_{k'_t}(P_{(f,t)}-\tilde
P_{(f,t)}: \hat X_p(\Omega) \rightarrow Y_q(\Omega))
\underset{\mathfrak{Z}_0} {\lesssim}
\\
\lesssim \sum \limits _{t=t_*(n)}^{t_{**}(n)-1} 2^{-t \left(\frac
1p-\frac 1q\right)} (2^{-t}n)^{\frac 1q-\frac 1p} \cdot 2^{-
c_*\cdot 2^{t-\varepsilon(t-t_*(n))}}
\underset{\mathfrak{Z}_0}{\lesssim}  n^{\frac 1q-\frac 1p}.
\end{array}
\end{align}

We claim that there exists a sequence $\{k''_t\}_{t_*'(n)\le t<
t_{**}(n)}$ such that
$$\sum \limits _{t=t_*'(n)}^{t_{**}(n)-1} (k''_t-1)
\underset{\mathfrak{Z}_0} {\lesssim} n,$$
\begin{align}
\label{sl_e} \sum \limits _{t=t_*'(n)}^{t_{**}(n)-1}
e_{k''_t}(\tilde P_{(f,t)}: \hat X_p(\Omega) \rightarrow
Y_q(\Omega))\underset{\mathfrak{Z}_0} {\lesssim} n^{\frac 1q-\frac
1p}.
\end{align}

For each $t_*'(n)\le t<t_{**}(n)$, $0\le l\le \log n_t$ we
consider the partition ${\bf T}_{f,t,l}$ as defined at Step 1. Let
${\cal D}\in {\bf T}_{f,t,l}$. We set $${\bf V}_{t,{\cal D}}^{*}
=\{\xi \in {\bf V}(\Gamma_t)\cap {\bf V}({\cal D}):\; {\bf
V}_1^{\cal A}(\xi) \cap {\bf V}(\Gamma_{t+1}) \ne \varnothing\}$$
and denote by ${\cal D}_+$ the subtree in ${\cal A}$ with vertex
set
\begin{align}
\label{vdplus} {\bf V}({\cal D}_+)= {\bf V}({\cal D}) \cup
\left(\cup _{\xi \in {\bf V}_{t,{\cal D}}^{*}} \cup _{\xi '\in
{\bf V}_1^{\cal A}(\xi) \cap {\bf V}(\Gamma_{t+1})}{\bf V}({\cal
A}_{\xi'})\right).
\end{align}
Let
\begin{align}
\label{ttftl} \tilde {\bf T}_{f,t,l}=\{{\cal D}_+:\; {\cal D} \in
{\bf T}_{f,t,l}\}.
\end{align}
By (\ref{ctftl}),
\begin{align}
\label{ctiltftl} {\rm card}\, \tilde {\bf T}_{f,t,l}
\underset{\mathfrak{Z}_0}{\lesssim} 2^{-l} n_t.
\end{align}

We claim that for any tree ${\cal D}\in {\bf T}_{f,t,l}$
\begin{align}
\label{cdpr} {\rm card}\, \{{\cal D}'_+:\; {\cal D}'\in {\bf
T}_{f,t,l\pm 1}:\; {\bf V}({\cal D}_+) \cap {\bf V}({\cal D}'_+)
\ne \varnothing\} \underset{\mathfrak{Z}_0}{\lesssim} 1.
\end{align}
Indeed, by (\ref{cabb}), it is sufficient to check that if ${\bf
V}({\cal D}_+) \cap {\bf V}({\cal D}'_+) \ne \varnothing$, then
${\bf V}({\cal D}) \cap {\bf V}({\cal D}') \ne \varnothing$. Let
$\xi \in {\bf V}({\cal D}_+) \cap {\bf V}({\cal D}'_+)$. Then
either $\xi \in {\bf V}(\hat{\cal A}_t)$ (therefore, $\xi \in {\bf
V}({\cal D}) \cap {\bf V}({\cal D}')$) or $\xi \in {\bf V}({\cal
A}_\eta)$, where $\eta \in {\bf V}(\Gamma_t) \cap {\bf V}({\cal
D}) \cap {\bf V}({\cal D}')$, ${\bf V}_1^{\cal A}(\eta) \cap {\bf
V}(\Gamma_{t+1}) \ne \varnothing$.

{\bf Assertion 2.} We have $\tilde {\bf T}_{f,t,0}|_{\hat
\Gamma_{f,t}}={\bf T}''_{f,t}$. If $({\cal T}, \, \hat \xi) \in
{\bf T}_{f,t,0}$, $\hat \xi \notin {\bf V}(\Gamma_t)$, then ${\bf
V}({\cal T}_+) \cap {\bf V}(\hat \Gamma_{f,t}) =\varnothing$.

{\it Proof of Assertion 2.}
\begin{enumerate}
\item Let $({\cal T}, \, \hat \xi) \in {\bf T}_{f,t,0}$.
We claim that either there exists a tree ${\cal D} \in {\bf
T}''_{f,t}$ such that ${\bf V}({\cal T}_+) \cap {\bf V}(\hat
\Gamma_{f,t}) ={\bf V}({\cal D})$ or ${\bf V}({\cal T}_+) \cap
{\bf V}(\hat \Gamma_{f,t}) =\varnothing$.

{\it Case $\hat \xi \in {\bf V}(\Gamma_t)$.} By (\ref{st_def}), we
have $\hat \xi \in S_t$. Let us show that
\begin{align}
\label{vtp} {\bf V}({\cal T}_+) \cap {\bf V}(\hat \Gamma_{f,t})
={\bf V}(\tilde {\cal D}_{(\hat \xi)})
\end{align}
and apply (\ref{tprpr_ft}).

\begin{itemize}
\item We claim that ${\bf V}({\cal T}_+) \cap {\bf V}(\hat \Gamma_{f,t})
\supset {\bf V}(\tilde {\cal D}_{(\hat \xi)})$. Since ${\bf
T}''_{f,t}$ is a partition of the graph $\hat \Gamma_{f,t}$, we
have ${\bf V}(\tilde {\cal D}_{(\hat \xi)}) \subset {\bf V}(\hat
\Gamma_{f,t})$. Let us prove that ${\bf V}(\tilde {\cal D}_{(\hat
\xi)}) \subset {\bf V}({\cal T}_+)$. Indeed, let $\xi \in {\bf
V}(\tilde {\cal D}_{(\hat \xi)}) \backslash {\bf V}({\cal T}_+)$.
We set $\eta_* =\min \{\eta\in [\hat \xi, \, \xi]:\; \eta \notin
{\bf V}({\cal T}_+)\}>\hat \xi$. Denote by $\zeta_*$ the direct
predecessor of $\eta_*$. Then $\zeta_* \notin \cup _{\zeta \in
{\bf V}_{t,{\cal T}}^{*}} \cup _{\zeta' \in {\bf V}_1^{\cal
A}(\zeta) \cap {\bf V}(\Gamma_{t+1})} {\bf V}({\cal A}_{\zeta'})$
(otherwise, $\eta\in {\bf V}({\cal T}_+)$). Hence, $\zeta_* \in
{\bf V}(\Gamma_t)$. If $\eta_*\notin {\bf V}(\Gamma_t)$, then
$\eta_*\in {\bf V}(\Gamma_{t+1})$ (by Assumption \ref{sup1},
condition 1); once again, we get $\eta_* \in {\bf V}({\cal T}_+)$.
Thus, $\eta_*\in {\bf V}(\Gamma_t)$ and $\eta_*$ is the minimal
vertex of some tree $\tilde {\cal T} \in {\bf T}_{f,t,0}$.
Therefore, $\eta_* \stackrel{(\ref{st_def})}{\in} S_t
\stackrel{(\ref{hst_def}), (\ref{hs_def})} {\subset} \hat S$,
which implies $\eta_* \stackrel{(\ref{vdxi}), (\ref{tilpft}),
(\ref{vtilde})}{\notin} {\bf V}(\tilde {\cal D}_{(\hat \xi)})$. On
the other hand, $\eta_*\in [\hat \xi, \, \xi] \subset {\bf
V}(\tilde {\cal D}_{(\hat \xi)})$, which leads to a contradiction.

\item We claim that ${\bf V}({\cal T}_+) \cap {\bf V}(\hat \Gamma_{f,t})
\subset {\bf V}(\tilde {\cal D}_{(\hat \xi)})$. We have ${\bf
V}(\hat \Gamma_{f,t}) \stackrel{(\ref{tprpr_ft})}{=}\cup _{\eta\in
S_t} {\bf V}(\tilde {\cal D}_{(\eta)})$. If $\zeta\in {\bf
V}(\tilde {\cal D}_{(\eta)}) \cap {\bf V}({\cal T}_+)$, then the
vertices $\hat \xi$ and $\eta$ are comparable. Since $\hat \xi \in
S_t \subset \hat S$, the case $\eta<\hat \xi$ is impossible by
(\ref{vdxi}), (\ref{tilpft}), (\ref{vtilde}). Consequently, $\eta
\in [\hat \xi, \, \zeta] \subset {\bf V}({\cal T}_+)$; in
addition, $\eta\in S_t$. By (\ref{st_def}), $\eta \in {\bf
V}(\Gamma_t)$ is the minimal vertex of some tree in ${\bf
T}_{f,t,0}$. Therefore, $\eta=\hat \xi$; i.e., $\zeta\in {\bf
V}(\tilde {\cal D}_{(\hat \xi)})$.
\end{itemize}
This completes the proof of (\ref{vtp}).

{\it Case $\hat \xi\notin {\bf V}(\Gamma_t)$.} Then $\hat \xi \in
{\bf V}(\Gamma_{t'})$, $t'<t$. Let us check that ${\bf V}({\cal
T}_+)\cap {\bf V}(\hat \Gamma_{f,t}) =\varnothing$. Indeed, let
$\xi \in {\bf V}({\cal T}_+)\cap {\bf V}(\hat \Gamma_{f,t})$. Then
$\xi \stackrel{(\ref{tprpr_ft})}{\in} {\bf V}(\tilde {\cal
D}_{(\eta)})$ for some $\eta\in S_t$; i.e., $\eta
\stackrel{(\ref{st_def})}{\in} {\bf V}(\Gamma_t)$ is the minimal
vertex of some tree from ${\bf T}_{f,t,0}$ (this tree does not
coincide with ${\cal T}$ since $\hat \xi \notin {\bf V}(\Gamma_t)$
is the minimal vertex of ${\cal T}$). In addition, the vertices
$\hat \xi$ and $\eta$ are comparable. The case $\eta<\hat \xi$ is
impossible by Assumption \ref{sup1} (see condition 1). Hence, $\eta\in
[\hat \xi, \, \xi] \subset {\bf V}({\cal T}_+)$; i.e., $\eta \in
{\bf V}({\cal T})$, which leads to a contradiction.

\item Let $\tilde{\cal D}_{(\hat \xi)}\in
{\bf T}''_{f,t}$. We claim that there exists a tree ${\cal T} \in
{\bf T}_{f,t,0}$ such that ${\bf V}(\tilde{\cal D}_{(\hat \xi)})
={\bf V}({\cal T}_+) \cap {\bf V}(\hat \Gamma_{f,t})$. Indeed,
since $\hat \xi\in S_t$ by (\ref{tprpr_ft}), we have $\hat \xi \in
{\bf V}(\Gamma_t)$ and $\hat \xi$ is the minimal vertex of some
tree  ${\cal T} \in {\bf T}_{f,t,0}$ (see (\ref{st_def})). By
(\ref{vtp}), ${\bf V}({\cal T}_+) \cap {\bf V}(\hat \Gamma_{f,t})=
{\bf V}(\tilde{\cal D}_{(\eta)})$ for some $\eta\in S_t$. In
addition, $\hat \xi \in {\bf V}({\cal T}_+) \cap {\bf V}(\hat
\Gamma_{f,t})$; i.e., $\hat \xi \in {\bf V}(\tilde{\cal
D}_{(\eta)})$ and $\eta =\hat \xi$.
\end{enumerate}
This completes the proof of Assertion 2. $\diamond$

\smallskip

Let ${\cal D}\in \tilde {\bf T}_{f,t,0}$, $t'_*(n)\le
t<t_{**}(n)$. If the minimal vertex of ${\cal D}$ does not belong
to ${\bf V}(\tilde \Gamma_{t_*'(n)})$, by Assertion 2 we have
${\bf V}({\cal D}) \cap {\bf V}(\hat \Gamma_{f,t}) =\varnothing$.
Hence, if ${\bf V}({\cal D}) \cap {\bf V}(\hat \Gamma_{f,t})
\ne\varnothing$, then $\tilde \Gamma_{t_*'(n)}\cap {\cal D}$ is a
tree.

Denote by $\tilde {\bf T}^n_{f,t,l}$ the partition of $\tilde
\Gamma_{t_*'(n)}$ formed by connected components of graphs $\tilde
\Gamma_{t_*'(n)}\cap {\cal D}$, ${\cal D}\in \tilde {\bf
T}_{f,t,l}$. We have
\begin{align}
\label{tpft} \tilde P_{(f,t)}h =P_{\tilde {\bf T}^n_{f,t,0}}h
\cdot\chi _{\Omega_{(f,t)}}, \quad h\in Y_q(\Omega).
\end{align}
It follows from (\ref{ptf}), (\ref{pft}), (\ref{tilpft}),
(\ref{tilpt}), (\ref{dop}), (\ref{p_ft_s}) and Assertion 2.

Given $0\le l\le \log n_t$, we set
\begin{align}
\label{ktl} \begin{array}{c} k_{t,l}=\lceil n\cdot
2^{-\varepsilon(t-t_*(n)+l)}\rceil, \quad t_*(n)\le t<t_{**}(n),
\\ k''_{t}=1+\sum \limits_{0\le l\le \log n_t} (k_{t,l}-1), \quad t_*'(n)\le t<t_{**}(n).
\end{array}
\end{align}
Then $k''_{t}-1 \underset{\mathfrak{Z}_0}{\lesssim} n\cdot
2^{-\varepsilon(t-t_*(n))}$, $\sum \limits
_{t=t_*'(n)}^{t_{**}(n)-1} (k''_t-1)
\underset{\mathfrak{Z}_0}{\lesssim} n$.

From (\ref{tftfl}), (\ref{vdplus}) and (\ref{ttftl}) it follows
that $\tilde {\bf T}_{f,t,\lfloor \log n_t \rfloor}=\{{\cal A}\}$,
$\tilde {\bf T}^n_{f,t,\lfloor \log n_t \rfloor}=\{{\cal
A}_{t_*'(n),j}\}_{j\in \hat J_{t_*'(n)}}$. From (\ref{qtf_x}) and
(\ref{ptf}) we get $P_{\tilde {\bf T}^n_{f,t,\lfloor \log n_t
\rfloor}}=0$. Hence, by (\ref{tpft}) and (\ref{eklspt}),
\begin{align}
\label{ekppt} \begin{array}{c} e_{k''_t}(\tilde P_{(f,t)}: \hat
X_p(\Omega) \rightarrow Y_q(\Omega)) \le \\ \le \sum \limits
_{0\le l< \lfloor\log n_t\rfloor} e_{k_{t,l}} (P_{\tilde {\bf
T}^n_{f,t,l}}-P_{\tilde {\bf T}^n_{f,t,l+1}}:\hat X_p(\Omega)
\rightarrow Y_q(\Omega_{(f,t)})).
\end{array}
\end{align}

Denote by $\hat {\bf T}^n_{f,t,l}$ the partition formed by trees
${\cal D}'\cap {\cal D}''$, where ${\cal D}'\in \tilde {\bf
T}^n_{f,t,l}$, ${\cal D}''\in \tilde {\bf T}^n_{f,t,l+1}$, and
either ${\cal D}'\in \tilde {\bf T}_{f,t,l}$ or ${\cal D}''\in
\tilde {\bf T}_{f,t,l+1}$. We set
\begin{align}
\label{tiltnf} \tilde T^n_{f,t,l}=\{\Omega_{{\cal D}}\}_{{\cal
D}\in \tilde {\bf T}^n_{f,t,l}}, \quad \hat
T_{f,t,l}^n=\{E=\Omega_{{\cal D}} \cap \Omega_{(f,t)}:\; {\rm
mes}\, E>0\}_{{\cal D}\in \hat {\bf T}^n_{f,t,l}}.
\end{align}

Let $h\in B\hat X_p(\Omega)$. We show that
\begin{align}
\label{ph_chi_oft} (P_{\tilde {\bf T}^n_{f,t,l}}h-P_{\tilde {\bf
T}^n_{f,t,l+1}}h)\chi_{\Omega_{(f,t)}} \in {\cal S}_{\hat
{T}^n_{f,t,l}}(\Omega).
\end{align}
Indeed, if ${\cal D}={\cal D}'\cap {\cal D}''$, ${\cal D}'\in
\tilde {\bf T}^n_{f,t,l}$, ${\cal D}''\in \tilde {\bf
T}^n_{f,t,l+1}$, then $P_{\tilde {\bf
T}^n_{f,t,l}}h|_{\Omega_{{\cal D}'}} \stackrel{(\ref{ptf})}{\in}
{\cal P}(\Omega_{{\cal D}'})$, $P_{\tilde {\bf
T}^n_{f,t,l+1}}h|_{\Omega_{{\cal D}''}}
\stackrel{(\ref{ptf})}{\in} {\cal P}(\Omega_{{\cal D}''})$;
therefore,
\begin{align}
\label{ptil} (P_{\tilde {\bf T}^n_{f,t,l}}h-P_{\tilde {\bf
T}^n_{f,t,l+1}}h)|_{\Omega_{{\cal D}} \cap \Omega_{(f,t)}} \in
{\cal P}(\Omega_{{\cal D}} \cap \Omega_{(f,t)}).
\end{align}
We show that if ${\cal D}' \notin \tilde {\bf T}_{f,t,l}$ and
${\cal D}''\notin \tilde {\bf T}_{f,t,l+1}$, then $(P_{\tilde {\bf
T}^n_{f,t,l}}h-P_{\tilde {\bf T}^n_{f,t,l+1}}h)|_{\Omega_{{\cal
D}} \cap \Omega_{(f,t)}}=0$. Indeed, in this case minimal vertices
of the trees ${\cal D}'$ and ${\cal D}''$ are equal and coincide
with $\hat \xi_{t_*(n),j}$ for some $j\in \hat J_{t_*(n)}$. From
(\ref{ptf}) it follows that $(P_{\tilde {\bf
T}^n_{f,t,l}}h-P_{\tilde {\bf T}^n_{f,t,l+1}}h)|_{\Omega_{{\cal
D}}}=0$. From (\ref{st_omega}) and (\ref{ptil}) we get
(\ref{ph_chi_oft}).

We claim that for any $E'\in \tilde T^n_{f,t,l}$, $E''\in \tilde
T^n_{f,t,l+1}$
\begin{align}
\label{ctl} \begin{array}{c} {\rm card}\, \{E\in \hat
T_{f,t,l}^n:\;  \; E \subset E'\}
\underset{\mathfrak{Z}_0}{\lesssim} 1, \quad {\rm card}\, \{E\in
\hat T_{f,t,l}^n:\;  \; E \subset E''\}
\underset{\mathfrak{Z}_0}{\lesssim} 1. \end{array}
\end{align}
Let us check the first inequality (the second one is proved
similarly). Let $E=\Omega_{{\cal D}'}$, ${\cal D}'\in \tilde {\bf
T}^n_{f,t,l}$, $E\in \hat T_{f,t,l}^n$, $E \subset E'$. Since
$\tilde {\bf T}^n_{f,t,l}$ is a partition, we have
$E=\Omega_{{\cal D}'\cap {\cal D}''} \cap \Omega_{(f, t)}$, ${\cal
D}''\in \tilde {\bf T}^n_{f,t,l+1}$, ${\bf V}({\cal D}') \cap {\bf
V}({\cal D}'') \ne \varnothing$. There exist trees ${\cal T}'\in
\tilde {\bf T}_{f,t,l}$ and ${\cal T}''\in \tilde {\bf
T}_{f,t,l+1}$ such that ${\cal D}'$ and ${\cal D}''$ are connected
components of the graphs ${\cal T}'\cap \tilde \Gamma_{t'_*(n)}$
and ${\cal T}''\cap \tilde \Gamma_{t'_*(n)}$, respectively.
Observe that the connected component of the graph ${\cal T}''\cap
\tilde \Gamma_{t'_*(n)}$ whose vertex set intersects with ${\bf
V}({\cal D}')$ is unique. This together with (\ref{cdpr}) implies
(\ref{ctl}).

From (\ref{ctiltftl}), (\ref{ctl}) and the definition of $\hat
{\bf T}^n_{f,t,l}$ we get that
\begin{align}
\label{card_htnftl} {\rm card}\,\hat T_{f,t,l}^n
\underset{\mathfrak{Z}_0}{\lesssim} 2^{-l} n_t.
\end{align}

For any $h\in B\hat X_p(\Omega)$ we have
$$
\|(P_{\tilde {\bf T}^n_{f,t,l}}h-P_{\tilde {\bf
T}^n_{f,t,l+1}}h)\chi _{\Omega_{(f,t)}}\| _{p,q, \hat T^n_{f,t,l}}
\stackrel{(\ref{fpqt}),
(\ref{ctl})}{\underset{\mathfrak{Z}_0}{\lesssim}}
$$
$$
\lesssim \|(h-Q_{t_*'(n)}h-P_{\tilde {\bf T}^n_{f,t,l}}h)\chi
_{\tilde U_{t_*'(n)}}\|_{p,q,\tilde T^n_{f,t,l}}+
$$
$$
+\|(h-Q_{t_*'(n)}h-P_{\tilde {\bf T}^n_{f,t,l+1}}h)\chi _{\tilde
U_{t_*'(n)}}\|_{p,q,\tilde T^n_{f,t,l+1}} \stackrel{(\ref{fpqt}),
(\ref{ptf}), (\ref{tiltnf})}{=}
$$
$$
\lesssim \left(\sum \limits _{({\cal D}, \, \xi)\in \tilde {\bf
T}^n_{f,t,l}} \|h-P_\xi h\|^p_{Y_q(\Omega_{{\cal
D}})}\right)^{1/p}+\left(\sum \limits _{({\cal D}, \, \xi)\in
\tilde {\bf T}^n_{f,t,l+1}} \|h-P_\xi h\|^p_{Y_q(\Omega_{{\cal
D}})}\right)^{1/p}
\stackrel{(\ref{f_pom_f_cor})}{\underset{\mathfrak{Z}_0}{\lesssim}}
$$
$$
\lesssim 2^{\left(\frac 1q-\frac 1p\right)t_*'(n)}
\underset{\mathfrak{Z}_0}{\lesssim} 2^{\left(\frac 1q-\frac
1p\right)t_*(n)}\stackrel{(\ref{2tn_est})}
{\underset{\mathfrak{Z}_0}{\asymp}} (\log n)^{\frac 1q-\frac 1p};
$$
i.e.,
\begin{align}
\label{ooo} \|(P_{\tilde {\bf T}^n_{f,t,l}}h-P_{\tilde {\bf
T}^n_{f,t,l+1}}h)\chi _{\Omega_{(f,t)}}\| _{p,q, \hat T^n_{f,t,l}}
\underset{\mathfrak{Z}_0}{\lesssim}(\log n)^{\frac 1q-\frac 1p}.
\end{align}

From (\ref{mult_n}), (\ref{ph_chi_oft}), (\ref{card_htnftl}),
(\ref{ooo}) and Lemma \ref{oper_a} we get that
\begin{align}
\label{ektl} \begin{array}{c} e_{k_{t,l}} (P_{\tilde {\bf
T}^n_{f,t,l}}-P_{\tilde {\bf T}^n_{f,t,l+1}}:\hat X_p(\Omega)
\rightarrow Y_q(\Omega_{(f,t)}))
\underset{\mathfrak{Z}_0}{\lesssim}
\\
\lesssim (\log n)^{\frac 1q-\frac
1p}e_{k_{t,l}}(I_{s_{t,l}}:l_p^{s_{t,l}} \rightarrow
l_q^{s_{t,l}})=:A_{t,l},
\end{array}
\end{align}
where $ s_{t,l}\in \N$,
$$
s_{t,l}\le C_*\cdot 2^{-l}n_t
\stackrel{(\ref{nt}),(\ref{sumet})}{\le}C_*\cdot 2^{-l}\lceil
n\cdot 2^{-t}\rceil , \quad C_*=C_*(\mathfrak{Z}_0)\ge 1.
$$

For each $t\le t_{**}(n)$ we have
$$
\frac{s_{t,l}}{k_{t,l}} \stackrel{ (\ref{ktl})} {\le}
\frac{C_*\cdot 2^{-l}\lceil n\cdot 2^{-t}\rceil}{ n\cdot
2^{-\varepsilon(l+t-t_*(n))}}=:\sigma'_{t,l},
$$
\begin{align}
\label{sigtl}
\sigma'_{t,l}\stackrel{(\ref{2tn_est})}{\underset{\mathfrak{Z}_0}{\asymp}}
\frac{2^{-l} n\cdot 2^{-t}}{n\cdot 2^{-\varepsilon(l+t-t_*(n))}}=
2^{-l(1-\varepsilon)-t(1-\varepsilon)-\varepsilon t_*(n)}\le 1.
\end{align}
The sequence $\left\{\sigma'_{t,l}\right\}_{l\in \Z_+}$ decreases
not slower than some geometric progression. This together with
Theorem \ref{shutt_trm} implies that there exists
$\gamma_0=\gamma_0(\mathfrak{Z}_0)>0$ such that
$$
\sum \limits _{t=t_*'(n)}^{t_{**}(n)-1} \sum \limits _{0\le l<
\log n_t} A_{t,l} \stackrel{(\ref{ektl})}{\underset
{\mathfrak{Z}_0}{\lesssim}}(\log
n)^{\frac 1q-\frac 1p} \sum \limits _{t=t_*'(n)}^{t_{**}(n)-1}
\sum \limits _{0\le l< \log n_t} s_{t,l}^{\frac 1q-\frac 1p}
2^{-\frac{k_{t,l}}{s_{t,l}}} {\underset{\mathfrak{Z}_0}{\lesssim}}
$$
$$
\lesssim (\log n)^{\frac 1q-\frac 1p} \sum \limits
_{t=t_*'(n)}^{t_{**}(n)-1} \sum \limits _{0\le l< \log n_t}
k_{t,l}^{\frac 1q-\frac 1p} \cdot (\sigma'_{t,l})^{\frac 1q-\frac
1p} \cdot 2^{-\frac{1}{\sigma'_{t,l}}}
\stackrel{(\ref{ktl}),(\ref{sigtl})}{\underset{\mathfrak{Z}_0}{\lesssim}}
$$
$$
\lesssim (\log n)^{\frac 1q-\frac 1p} \sum \limits
_{t=t_*(n)}^{t_{**}(n)-1} k_{t,0}^{\frac 1q-\frac
1p}\cdot2^{\left(t(1-\varepsilon)+\varepsilon
t_*(n)\right)\left(\frac 1q-\frac 1p\right)} \cdot
2^{-\gamma_0\cdot2^{t(1-\varepsilon)+\varepsilon t_*(n)}}
\stackrel{(\ref{ktl})}{\underset{\mathfrak{Z}_0}{\lesssim}}
$$
$$
\lesssim (\log n)^{\frac 1q-\frac 1p} \cdot n^{\frac 1q-\frac
1p}\cdot 2^{\left(\frac 1q-\frac 1p\right)t_*(n)} \cdot
2^{-\gamma_0\cdot 2^{t_*(n)}}
\stackrel{(\ref{2tn_est})}{\underset{\mathfrak{Z}_0}{\lesssim}}
n^{\frac 1q-\frac 1p};
$$
i.e.,
\begin{align}
\label{sumatl} \sum \limits _{t=t'_*(n)}^{t_{**}(n)-1} \sum
\limits _{0\le l< \log n_t} A_{t,l}
\underset{\mathfrak{Z}_0}{\lesssim} n^{\frac 1q-\frac 1p}.
\end{align}

From (\ref{ekppt}), (\ref{ektl}) and (\ref{sumatl}) we get
(\ref{sl_e}). This together with (\ref{sum_e_p}) yields
(\ref{slttn}). Applying (\ref{sum_ptft}), we have (\ref{enptf}).
Taking into account the estimate obtained at Step 5, we complete
the proof of Lemma \ref{main_lem}.
\end{proof}

It remains to prove that
\begin{align}
\label{enqnm} e_n \left(\sum \limits _{m=0}^\infty (\tilde
Q_{n,m+1} -\tilde Q_{n,m}):\hat X_p(\Omega) \rightarrow
Y_q(\Omega)\right) \underset{\mathfrak{Z}_0}{\lesssim} n^{\frac
1q-\frac 1p}.
\end{align}

In \cite{kuhn5} there were obtained order estimates for entropy
numbers of diagonal operators with weights of logarithmic type.
First we give some notations.

Denote by $\Phi_0$ the class of non-decreasing functions $\varphi:
[1, \, \infty) \rightarrow (0, \, \infty)$ that satisfy the
following condition: there exist $c>0$ and $\alpha>0$ such that
for any $1\le s\le t<\infty$
\begin{align}
\label{phi_ts} \frac{\varphi(t)}{\varphi(s)} \le
c\left(\frac{1+\log t}{1+\log s}\right)^\alpha.
\end{align}
We set $w_\varphi(s)=1$ for $0\le s\le 1$ and
$w_\varphi(s)=\varphi(s)$ for $s>1$.

Let $\delta>0$, $w:\R_+ \rightarrow \R_+$ be a continuous
function, and let $0<r, \, p\le \infty$. For $x=(x_{m,k})_{m,k\in
\Z_+}$ we set
$$
\|x|l_r(2^{\delta m}l_p(w))\|:=\left(\sum \limits _{m=0}^\infty
2^{m\delta r}\left(\sum \limits _{k\in \Z_+}
|x_{m,k}w(2^{-m}k)|^p\right)^{\frac rp}\right)^{\frac 1r}
$$
(appropriately modified if $p=\infty$ or $r=\infty$). By
$l_r(2^{\delta m}l_p(w))$ we denote the space of sequences $x$
such that $\|x|l_r(2^{\delta m}l_p(w))\|<\infty$.

\begin{trma}
\label{kuhn_seq} {\rm \cite[p. 11]{kuhn5}.} Let $0<p<q\le \infty$,
$0<r, \, s\le \infty$, $\delta>0$, $\varphi\in \Phi_0$, and let
(\ref{phi_ts}) hold with $\alpha =\frac 1p-\frac 1q$. Then
$$
e_n({\rm id}: l_r(2^{\delta m}l_p(w_\varphi)) \rightarrow
l_s(l_q)) \underset{p,q,r,s,\varphi}{\asymp}
\frac{1}{\varphi(2^n)}.
$$
\end{trma}

Let us prove (\ref{enqnm}).

Given $m\ge m_t$, we denote by $\hat T_{t,m}$ the partition of the
set $G_t$ formed by $E'\cap E''$, $E'\in \tilde T_{t,m}$, $E''\in
\tilde T_{t,m+1}$ (the partitions $\tilde T_{t,m}$ are defined at
page \pageref{tttm}). Let $s''_{t,m}=\dim {\cal S}_{\hat
T_{t,m}}(\Omega)$. Then there exists $M=M(\mathfrak{Z}_0)\ge 1$
such that
\begin{align}
\label{spptm} s''_{t,m} \stackrel{(\ref{nu_t_k1}), (\ref{mt_def}),
(\ref{cttm}), (\ref{card_et})}{\le} M\cdot 2^{m-m_t}\cdot
2^{\gamma_*2^t} \psi_*(2^{2^t}).
\end{align}
Given $t\ge 0$, $m'\ge 0$, we set
$$
\hat s_{t,m'} =\lceil M\cdot 2^{m'}\cdot 2^{\gamma_*2^t}
\psi_*(2^{2^t})\rceil, \quad s^*_{t,m'} =\sum \limits _{l=0}^{t-1}
\hat s_{l,m'}.
$$

Denote $\varphi(x)=\left(\log (2+x)\right)^{\frac 1p-\frac 1q}$,
$x\ge 0$. Let $k=s^*_{t,m'}+j$, $1\le j\le \hat s_{t,m'}$. Then
$$
2^{\gamma_*2^{t-1}} \psi_*(2^{2^{t-1}}) \underset{\mathfrak{Z}_0}
{\lesssim} 2^{-m'}k \underset{\mathfrak{Z}_0}{\lesssim}
2^{\gamma_*2^t} \psi_*(2^{2^t}), \quad t\ge 1,
$$
$2^{-m'}k\underset{\mathfrak{Z}_0}{\lesssim} 1$ if $t=0$. Hence,
\begin{align}
\label{phi_est} w_{\varphi}(2^{-m'}k)
\stackrel{(\ref{te})}{\underset{\mathfrak{Z}_0}{\asymp}}
2^{\left(\frac 1p-\frac 1q\right)t}, \quad k=s^*_{t,m'}+j, \quad
1\le j\le \hat s_{t,m'}.
\end{align}

From Lemma \ref{oper_a} it follows that there is an isomorphism
$\overline{A}_{t,m}:{\cal S}_{\hat T_{t,m}}(\Omega) \rightarrow
\R^{s''_{t,m}}$ such that
\begin{align}
\label{mtmn} \|\overline{A}_{t,m}\|_{Y_{p,q,\hat
T_{t,m}}(G_t)\rightarrow l_p^{s''_{t,m}}}
\underset{\mathfrak{Z}_0}{\lesssim} 1, \quad
\|\overline{A}_{t,m}^{\, -1}\|_{l_q^{s''_{t,m}}\rightarrow
Y_q(G_t)}\underset{\mathfrak{Z}_0}{\lesssim} 1.
\end{align}

Let us define the operator $\overline{A}:\hat X_p(\Omega)
\rightarrow l_\infty (2^{\delta_*m'} l_p(w_\varphi))$ as follows.
Consider a function $f\in \hat X_p(\Omega)$. Then $(\tilde
Q_{n,m'+1}f-\tilde Q_{n,m'}f)|_{G_t}\stackrel {(\ref{ptm_yq}),
(\ref{qnm_def})}{\in} {\cal S}_{\hat T_{t,m_t+m'}}(\Omega)$. Let
$$
\overline{A}_{t,m_t+m'}((\tilde Q_{n,m'+1}f-\tilde
Q_{n,m'}f)|_{G_t}) =(c_{m',t,j})_{j=1}^{s''_{t,m_t+m'}}.
$$
We set
$$
(\overline{A}f)_{m',s^*_{t,m'}+j} =\left\{
\begin{array}{l} c_{m',t,j} \quad \text{for} \quad 1\le j\le s''_{t,m_t+m'},\; t\ge t_0, \\
0 \quad \text{for}\quad s''_{t,m_t+m'}+1\le j\le \hat
s_{t,m'}\text{ or }t<t_0.\end{array} \right.
$$
Then
$$
\|\overline{A}f\|_{l_\infty (2^{\delta_*m'} l_p(w_\varphi))} =
\sup _{m'\ge 0} 2^{\delta_* m'} \left( \sum \limits _{k\in \Z_+}
|w_\varphi(2^{-m'}k)(\overline{A}f)_{m',k} |^p\right)^{\frac 1p}
\stackrel{(\ref{phi_est})}{\underset{\mathfrak{Z}_0}{\lesssim}}
$$
$$
\lesssim \sup _{m'\ge 0} 2^{\delta_*m'} \left( \sum \limits _{t\ge
t_0} 2^{\left(1-\frac pq\right)t}\sum \limits
_{j=1}^{s''_{t,m_t+m'}} |c_{m',t,j}|^p\right)^{1/p}=: N.
$$
From (\ref{fpttm1}), (\ref{fpttm2}) and (\ref{card_et}) it follows
that for $m\ge m_t$
$$
\|P_{t,m+1}-P_{t,m}\|_{\hat X_p(G_t) \rightarrow Y_{p,q,\hat
T_{t,m}}(G_t)} \underset{\mathfrak{Z}_0}{\lesssim}
2^{-\delta_*(m-m_t)} \cdot 2^{\left(\frac 1q-\frac 1p\right)t}
\|f\|_{X_p(G_t)}.
$$
Hence, by (\ref{qnm_def}), (\ref{mtmn}) we get $$\sum \limits
_{j=1}^{s''_{t,m_t+m'}} |c_{m',t,j}|^p
\underset{\mathfrak{Z}_0}{\lesssim} 2^{-\delta_*pm'}
2^{\left(\frac pq-1\right)t} \|f\|^p_{X_p(G_t)};$$ therefore,
$$
N \underset{\mathfrak{Z}_0}{\lesssim} \left(\sum \limits _{t\ge
t_0} \|f\|^p_{X_p(G_t)}\right)^{1/p} =\|f\|_{X_p(\Omega)}.
$$
Thus,
\begin{align}
\label{a_norm} \|\overline{A}\|_{\hat X_p(\Omega) \rightarrow
l_\infty (2^{\delta_*m'} l_p(w_\varphi))}
\underset{\mathfrak{Z}_0}{\lesssim} 1.
\end{align}

Let us define the operator $\overline{K}: l_1(l_q) \rightarrow
Y_q(\Omega)$ by formula
$$
\overline{K}((c_{m',s^*_{t,m'}+j})_{m'\in \Z_+,t\in \Z_+, 1\le
j\le \hat s _{t,m'}}) =\sum \limits _{m'\in \Z_+ } \sum \limits
_{t\ge t_0} \overline{A}^{\, -1}_{t,m_t+m'}
((c_{m',s^*_{t,m'}+j})_{j=1}^{s''_{t,m_t+m'}}).
$$
Since the sets $G_t$ do not overlap pairwise, we have
$$
\|\overline{K}((c_{m',s^*_{t,m'}+j})_{m'\in \Z_+,t\in \Z_+, 1\le
j\le \hat s _{t,m'}})\| _{Y_q(\Omega)} \le
$$
$$
\le \sum \limits _{m'=0}^\infty \left( \sum \limits
_{t=t_0}^\infty \|\overline{A}^{\, -1}_{t,m_t+m'}
((c_{m',s^*_{t,m'}+j})_{j=1}^{s''_{t,m_t+m'}})\|^q_{Y_q(G_t)}\right)^{1/q}
\stackrel{(\ref{mtmn})}{\underset{\mathfrak{Z}_0}{\lesssim}}
$$
$$
\lesssim \sum \limits _{m'=0}^\infty \left(\sum \limits _{t\in
\Z_+} \sum \limits _{j=1}^{\hat s_{t,m'}}
|c_{m',s^*_{t,m'}+j}|^q\right)^{1/q} =
\|(c_{m',s^*_{t,m'}+j})_{m'\in \Z_+,t\in \Z_+, 1\le j\le \hat s
_{t,m'}}\|_{l_1(l_q)}.
$$
Hence,
\begin{align}
\label{kn} \|\overline{K}\|_{l_1(l_q) \rightarrow Y_q(\Omega)}
\underset{\mathfrak{Z}_0}{\lesssim} 1.
\end{align}

Let ${\rm id}: l_\infty(2^{\delta_* m'}l_p(w_\varphi)) \rightarrow
l_1(l_q)$ be the identity operator. Then
$$
e_n \left(\sum \limits _{m=0}^\infty (\tilde Q_{n,m+1} -\tilde
Q_{n,m}):\hat X_p(\Omega) \rightarrow Y_q(\Omega)\right)=
$$
$$
=e_n({\overline{K}}\circ {\rm id} \circ \overline{A}:\hat
X_p(\Omega) \rightarrow Y_q(\Omega))
\stackrel{(\ref{mult_n}),(\ref{a_norm}),(\ref{kn})}{\underset{\mathfrak{Z}_0}{\lesssim}}
$$
$$
\lesssim e_n({\rm id}: l_\infty(2^{\delta_* m'}l_p(w_\varphi))
\rightarrow l_1(l_q))=:E_n.
$$
From Theorem \ref{kuhn_seq} we get $E_n
\underset{\mathfrak{Z}_0}{\lesssim} n^{\frac 1q-\frac 1p}$.

Applying Lemmas \ref{sum_qt_est}, \ref{ptm_sum}, \ref{main_lem}
together with (\ref{f_razl}), (\ref{fqf}), (\ref{fmqtsnf}),
(\ref{enqnm}), we complete the proof of Theorem \ref{main_tree}.

\begin{Rem}
\label{discr_rem0} Suppose that instead of Assumption \ref{sup2}
the following condition holds: for any $\xi\in {\bf V}({\cal A})$
the set $\hat F(\xi)$ is the atom of ${\rm mes}$. Then the
assertion of Theorem \ref{main_tree} holds as well.
\end{Rem}

\renewcommand{\proofname}{\bf Proof}

\section{Some particular cases}

Let $({\cal T}, \xi_0)$ be a tree, and let $u$, $w:{\bf V}({\cal
T}) \rightarrow \R_+$. We define the summation operator
$S_{u,w,{\cal T}}$ by
$$
S_{u,w,{\cal T}}f(\xi) = w(\xi)\sum \limits _{\xi'\le \xi}
u(\xi')f(\xi'), \quad \xi \in {\bf V}({\cal T}), \quad f:{\bf
V}({\cal T}) \rightarrow \R.
$$
Let $1\le p, \, q\le \infty$. By
\label{obozn_spqguw}$\mathfrak{S}^{p,q}_{{\cal T},u,w}$ we denote
the minimal constant $C$ in the inequality
$$
\left(\sum \limits_{\xi \in {\bf V}({\cal T})} w^q(\xi) \left|
\sum \limits _{\xi'\le \xi}u(\xi')f(\xi')\right|^q\right)^{1/q}
\le C\left(\sum \limits_{\xi \in {\bf V}({\cal T})}
|f(\xi)|^p\right)^{1/p}, \;\; f:{\bf V}({\cal T})\rightarrow \R
$$
(appropriately modified for $p=\infty$ or $q=\infty$).

Given a tree $({\cal T}, \xi_0)$, $N\in \Z_+$, we denote by
$[{\cal T}]_{\le N}$ a tree with vertex set $\cup_{j=0}^N {\bf
V}_j^{{\cal T}}(\xi_0)$.

Let $({\cal A}, \, \xi_0)$ be a tree, let (\ref{c_v1_a}) hold, and
let the measure space $(\Omega, \, \Sigma, \, {\rm mes})$, the
partition $\Theta$, the bijection $\hat F:{\bf V}({\cal A})
\rightarrow \hat \Theta$ and the spaces $X_p(\Omega)$,
$Y_q(\Omega)$, ${\cal P}(\Omega)$ be as defined at the page
\pageref{xpyq}. Assumptions \ref{sup1}--\ref{sup3} will be
replaced by the following conditions.
\begin{Supp}
\label{supp1} There exist functions $u$, $w: {\bf V}({\cal A})
\rightarrow (0, \, \infty)$ and a constant $c_2\ge 1$ with the
following property: for any vertex $\xi_*\in {\bf V}({\cal A})$
there exists a linear continuous projection $P_{\xi_*}:
Y_q(\Omega)\rightarrow {\cal P}(\Omega)$ such that for any vertex
$\xi\ge \xi_*$ and for any function $f\in X_p(\Omega)$
\begin{align}
\label{hfxi} \|f-P_{\xi_*}f\|_{Y_q(\hat F(\xi))} \le c_2 w(\xi)
\sum \limits _{\xi_*\le \xi'\le \xi} u(\xi') \|f\|_{X_p(\hat
F(\xi'))}.
\end{align}
\end{Supp}

\begin{Supp}
\label{supp2} There exists a number $\delta_*>0$ such that for any
vertex $\xi\in {\bf V}({\cal A})$ and for any $n\in \N$, $m\in
\Z_+$ there exists a partition $T_{m,n}(G)$ of the set $G=\hat
F(\xi)$ with the following properties:
\begin{enumerate}
\item ${\rm card}\, T_{m,n}(G)\le c_2\cdot 2^mn$.
\item For any $E\in T_{m,n}(G)$ there exists a linear continuous
operator $P_E:Y_q(\Omega)\rightarrow {\cal P}(E)$ such that for
any function $f\in X_p(\Omega)$
\begin{align}
\label{fpef1} \|f-P_Ef\|_{Y_q(E)}\le c_2(2^mn)^{-\delta_*}
u(\xi)w(\xi) \|f\| _{X_p(E)}.
\end{align}
\item For any $E\in T_{m,n}(G)$
\begin{align}
\label{ceptm1} {\rm card}\,\{E'\in T_{m\pm 1,n}(G):\, {\rm
mes}(E\cap E') >0\} \le c_2.
\end{align}
\end{enumerate}
\end{Supp}

Let ${\bf V}({\cal A})=\{\eta_{j,i}\}_{j\ge j_{\min}, \, i\in
\tilde I_j}$, where $j_{\min}\ge 0$, $\tilde I_{j_{\min}}=\{1\}$.
Suppose that $\eta_{j_{\min},1}$ is the minimal vertex of ${\cal
A}$ and ${\bf V}^{{\cal
A}}_{j-j_{\min}}(\eta_{j_{\min},1})=\{\eta_{j,i}\}_{i\in \tilde
I_j}$ for any $j\ge j_{\min}$.

\begin{Supp}
\label{supp3} There exist numbers $\theta>0$, $\gamma\in \R$,
$\kappa\ge \frac{\theta}{q}$, $\alpha_u$, $\alpha_w\in \R$,
$c_3\ge 1$, $m_*\in \N$ and an absolutely continuous function
$\tau:(0, \, \infty) \rightarrow (0, \, \infty)$ such that $\lim
\limits _{t\to +\infty} \frac{t\tau'(t)}{\tau(t)}=0$ and the
following conditions hold.
\begin{enumerate}
\item If $\kappa=\frac{\theta}{q}$, then $\alpha_w>\frac{1-\gamma}{q}$.
\item If $\kappa> \frac{\theta}{q}$, then $\alpha_u+\alpha_w=
\frac 1p-\frac 1q$; if $\kappa= \frac{\theta}{q}$, then
$\alpha_u+\alpha_w=\frac 1p$.
\item For any $j'\ge j\ge j_{\min}$ and for any vertex $\xi \in {\bf V}
_{j-j_{\min}}^{{\cal A}}(\eta_{j_{\min},1})$
\begin{align}
\label{card_jj} {\rm card}\, {\bf V}_{j'-j}^{{\cal A}}(\xi) \le
c_3 \cdot 2^{\theta m_*(j'-j)} \frac{j^\gamma
\tau(m_*j)}{j'^\gamma \tau(m_*j')}.
\end{align}
\item For any $j\ge j_{\min}$, $i\in \tilde I_j$
\begin{align}
\label{uetaji} u(\eta_{j,i}) =u_j= 2^{\kappa
m_*j}(m_*j)^{-\alpha_u},\quad w(\eta_{j,i}) =w_j=2^{-\kappa m_*j}
(m_*j)^{-\alpha_w}.
\end{align}
\item Let $\kappa=\frac{\theta}{q}$. Then there exists a tree $\hat {\cal
A}$ with the minimal vertex $\hat \zeta_0$, which satisfies the
following conditions.
\begin{enumerate}
\item For any $j'\ge j\ge j_{\min}$ and for any vertex $\xi \in {\bf V}
_{j-j_{\min}}^{\hat {\cal A}}(\hat\zeta_0)$
\begin{align}
\label{card_jj1} c_3^{-1} \cdot 2^{\theta m_*(j'-j)}
\frac{j^\gamma \tau(m_*j)}{j'^\gamma \tau(m_*j')}\le {\rm card}\,
{\bf V}_{j'-j}^{\hat{\cal A}}(\xi) \le c_3 \cdot 2^{\theta
m_*(j'-j)} \frac{j^\gamma \tau(m_*j)}{j'^\gamma \tau(m_*j')}.
\end{align}
\item Let $\{\overline{u}_j\}_{j\ge j_{\min}}\subset (0, \,
\infty)$, $\{\overline{w}_j\}_{j\ge j_{\min}}\subset (0, \,
\infty)$ be arbitrary sequences. Define the functions $\tilde u$,
$\tilde w:{\bf V}({\cal A}) \rightarrow (0, \, \infty)$, $\hat u$,
$\hat w: {\bf V}(\hat{\cal A}) \rightarrow (0, \, \infty)$ by
$$
\tilde u|_{{\bf V}^{\cal A}_{j-j_{\min}}(\eta_{j_{\min},1})}
\equiv \overline{u}_j, \quad \hat u|_{{\bf V}^{\hat{\cal
A}}_{j-j_{\min}}(\zeta_0)} \equiv \overline{u}_j,
$$
$$
\tilde w|_{{\bf V}^{\cal A}_{j-j_{\min}}(\eta_{j_{\min},1})}
\equiv \overline{w}_j, \quad \hat w|_{{\bf V}^{\hat{\cal
A}}_{j-j_{\min}}(\zeta_0)} \equiv \overline{w}_j, \quad j\ge
j_{\min}.
$$
Then for each $N\ge j\ge j_{\min}$, $i\in \tilde I_j$ there exists
a vertex $\hat \xi\in {\bf V}^{\hat{\cal A}}_{j-j_{\min}}(\hat
\zeta_0)$ such that $$\mathfrak{S}^{p,q}_{[{\cal
A}_{\eta_{j,i}}]_{\le N-j},\tilde u,\tilde w} \le c_3
\mathfrak{S}^{p,q}_{[\hat{\cal A}_{\hat \xi}]_{\le N-j},\hat
u,\hat w}.$$
\end{enumerate}
\end{enumerate}
\end{Supp}

Denote $\mathfrak{Z}_0=(p, \, q, \, c_1, \, c_2, \, c_3, \,
\theta, \, \gamma, \, \kappa, \, \alpha_u, \, \alpha _w, \, m_*,
\, \delta_*, \, \tau)$.

Let $\xi_*\in {\bf V}({\cal A})$, and let ${\cal D} \subset {\cal
A}$ be a subtree with the minimal vertex $\xi_*$. Then from
Assumption \ref{supp1} it follows that
\begin{align}
\label{fpxif_spq} \|f-P_{\xi_*}f\|_{Y_q(\Omega_{\cal D})} \le c_2
\mathfrak{S}^{p,q} _{{\cal D},u,w} \|f\|_{X_p(\Omega_{{\cal D}})}.
\end{align}
From Theorem F in \cite{vas_width_raspr}, Theorem 3.6 in
\cite{vas_har_tree} and Assumption \ref{supp3} it follows that if
$\xi_*\in {\bf V}^{{\cal A}}_{j-j_{\min}}(\xi_0)$, then in the
case $\kappa
>\frac{\theta}{q}$
\begin{align}
\label{spq_duw} \mathfrak{S}^{p,q} _{{\cal D},u,w}
\underset{\mathfrak{Z}_0}{\lesssim} \sup _{s\ge j} u_sw_s
\underset{\mathfrak{Z}_0}{\asymp} (m_*j)^{-\alpha_u-\alpha_w}
=(m_*j)^{\frac 1q-\frac 1p},
\end{align}
and in the case $\kappa=\frac{\theta}{q}$
\begin{align}
\label{spq_duw1} \displaystyle \begin{array}{c} \mathfrak{S}^{p,q}
_{{\cal D},u,w} \underset{\mathfrak{Z}_0}{\lesssim} \sup _{s\ge j}
\left(\sum \limits _{i=j}^s (m_*i)^{-p'\alpha_u} \cdot 2^{p'\kappa
m_*i}\right)^{\frac{1}{p'}} \left(\sum \limits _{i\ge s}
(m_*i)^{-\alpha_wq}\cdot\frac{s^\gamma \tau(m_*s)}{i^\gamma
\tau(m_*i)} \cdot 2^{-q\kappa
m_*s}\right)^{\frac 1q} \underset{\mathfrak{Z}_0}{\lesssim} \\
\lesssim (m_*j)^{-\alpha_u-\alpha_w+\frac 1q} =(m_*j)^{\frac
1q-\frac 1p}.
\end{array}
\end{align}
Notice that the proof of (\ref{spq_duw1}) depends on condition 5
of Assumption \ref{supp3}.

From Assumption \ref{supp2} and conditions 2, 4 of Assumption
\ref{supp3} it follows that if $\xi \in {\bf V}_{j-j_{\min}}^{\cal
A}(\xi_0)$, $G=\hat F(\xi)$, $E\in T_{m,n}(G)$, then
\begin{align}
\label{fpefyq} \begin{array}{c} \|f-P_Ef\|_{Y_q(E)} \le
c_2\cdot(2^mn)^{-\delta_*} (m_*j)^{-\alpha_u-\alpha_w} \|f\|
_{X_p(E)} \le \\ \le c_2\cdot (2^mn)^{-\delta_*} (m_*j)^{\frac
1q-\frac 1p} \|f\| _{X_p(E)}. \end{array}
\end{align}

We construct the partition $\{{\cal A}_{t,i}\}_{t\ge t_0, \, i\in
\hat J_t}$ as follows. Let $$t_0=\min \{t\in \Z_+:\; 2^{t}> m_*
j_{\min}\}.$$ Given $t\ge t_0$, we denote by $\Gamma_t$ the
maximal subgraph in ${\cal A}$ on the vertex set
\begin{align}
\label{v_of_gamma_t} {\bf V}(\Gamma_t)=\{\eta_{j,s}:\; 2^{t-1}\le
m_*j<2^{t}, \; s\in \tilde I_j\},
\end{align}
and by ${\cal A}_{t,i}$, $i\in \hat J_t$, the connected components
of the graph $\Gamma_t$. By $\hat \xi_{t,i}$ we denote the minimal
vertex of the tree ${\cal A}_{t,i}$. Then
\begin{align}
\label{c_vgt} {\rm card}\, {\bf V}(\Gamma_t)
\stackrel{(\ref{card_jj})}{\underset{\mathfrak{Z}_0}{\lesssim}}
2^{\theta \cdot 2^t}2^{-\gamma t} \tau^{-1}(2^t).
\end{align}
Thus, Assumption \ref{sup3} holds with $\overline{\nu}_t
=2^{\theta \cdot 2^t}2^{-\gamma t} \tau^{-1}(2^t)$. From
(\ref{fpefyq}) we obtain that Assumption \ref{sup2} holds.

Recall the notations $J_{t,{\cal D}}=\{i\in \hat J_t:\; {\bf
V}({\cal A}_{t,i})\cap {\bf V}({\cal D}) \ne \varnothing\}$,
${\cal D}_{t,i} ={\cal D}\cap {\cal A}_{t,i}$, where ${\cal
D}\subset {\cal A}$ is a subtree.

In order to obtain Assumption \ref{sup1}, it is sufficient to
prove the following assertion.
\begin{Lem}
\label{lem_est} Let ${\cal D}$ be a subtree in ${\cal A}$ rooted
at $\xi_*$. Then
\begin{align}
\label{fpxifd} \|f-P_{\xi_*}f\|_{Y_q(\Omega_{\cal D})}^q
\underset{\mathfrak{Z}_0}{\lesssim} \sum \limits _{t=t_0}
^{\infty} 2^{\left(1-\frac qp\right) t} \sum \limits _{i\in J_{t,
{\cal D}}} \|f\|_{X_p(\Omega _{{\cal D}_{t,i}})} ^q.
\end{align}
\end{Lem}
\begin{proof}
By (\ref{hfxi}),
$$
\|f-P_{\xi_*}f\|^q_{Y_q(\Omega_{\cal D})} \underset
{\mathfrak{Z}_0} {\lesssim} \sum \limits _{\xi \in {\bf V}({\cal
D})} w^q(\xi) \left(\sum \limits _{\xi_*\le \xi'\le \xi} u(\xi')
\|f\|_{X_p(\hat F(\xi'))}\right)^q.
$$
Let $\kappa >\frac{\theta}{q}$. Then $\alpha_u+\alpha_w=\frac
1p-\frac 1q$ (see condition 2 of Assumption \ref{supp3}).
Repeating the proof of Lemma 5.1 in \cite{vas_vl_raspr2}, we get
that
$$
\|f-P_{\xi_*}f\|^q_{Y_q(\Omega_{\cal D})} \underset
{\mathfrak{Z}_0} {\lesssim} \sum \limits _{\xi \in {\bf V}({\cal
D})} w^q(\xi) u^q(\xi) \|f\|_{X_p(\hat F(\xi))}^q.
$$
If $\xi=\eta_{j,i}$, $2^{t-1}\le m_*j<2^t$, then $u^q(\xi)w^q(\xi)
\stackrel{(\ref{uetaji})}{=} (m_*j)^{-\frac qp+1}
\underset{\mathfrak{Z}_0}{\asymp} 2^{\left(1-\frac qp\right) t}$.
This together with the condition $p<q$ implies (\ref{fpxifd}).

Now we consider the case $\kappa=\frac{\theta}{q}$.

Let $\xi_*\in {\bf V} ({{\cal A}}_{\hat t,j_0})$, and let $\xi \in
{\bf V}({\cal D}_{\hat t+l,j_l})$. Then there exists a sequence
$\{\hat \xi_{\hat t+s,j_s}\}_{s=1}^l$ such that $\xi_*<\hat
\xi_{\hat t+1,j_{1}}< \hat \xi_{\hat t+2,j_{2}}< \dots < \hat
\xi_{\hat t+l,j_{l}}$. Denote
\begin{align}
\label{txi_def} \tilde \xi_{\hat t,j_0}=\xi_*, \quad \tilde
\xi_{\hat t+s,j_s} =\hat \xi_{\hat t+s,j_s}, 1\le s\le l.
\end{align}
We have
\begin{align}
\label{ufxi_sum} \sum \limits _{\xi_*\le \xi'\le \xi} u(\xi')
f(\xi') = \sum \limits _{s=0}^{l-1} \sum \limits _{\tilde
\xi_{\hat t+s,j_s} \le \xi' < \tilde \xi_{\hat t+s+1,j_{s+1}}}
u(\xi') \|f\|_{X_p(\hat F(\xi'))} + \sum \limits _{\tilde
\xi_{\hat t+l,j_l}\le \xi'\le \xi} u(\xi') \|f\|_{X_p(\hat
F(\xi'))}.
\end{align}

Let $\tilde \alpha_w=\alpha_w-\frac {1}{q}$. By condition 2 of
Assumption \ref{supp3},
\begin{align}
\label{auaw} \alpha_u+\tilde \alpha_w=\frac 1p-\frac 1q.
\end{align}

We have
$$
\sum \limits _{\xi \in {\bf V}({\cal D}_{\hat t+l, \, j_l})}
w^q(\xi) \stackrel{(\ref{card_jj}), (\ref{uetaji}),
(\ref{v_of_gamma_t})}{\underset{\mathfrak{Z}_0}{\lesssim}}$$$$\lesssim
\sum \limits _{2^{\hat t+l-1} \le m_*j< 2^{\hat t+l}} 2^{-\theta
m_*j} (m_*j)^{-\alpha_w q} \cdot 2^{\theta (m_*j-2^{\hat t+l-1})}
\left(\frac{2^{\hat t+l-1}}{m_*j}\right)^{\gamma}
\frac{\tau(2^{\hat t+l-1})}{\tau(m_*j)}
\stackrel{(\ref{te})}{\underset{\mathfrak{Z}_0}{\lesssim}}
$$
$$
\lesssim 2^{-\kappa q\cdot 2^{\hat t+l-1}} \cdot 2^{(\hat t+l)
(-\alpha_w q+1)}=2^{-\kappa q\cdot 2^{\hat t+l-1}} \cdot 2^{-(\hat
t+l) \tilde\alpha_w q},
$$
which implies
$$
\|f-P_{\xi_*}f\| _{Y_q(\Omega_{{\cal D}_{\hat t+l,j_l}})}^q
\stackrel{(\ref{hfxi})}{\underset{\mathfrak{Z}_0}{\lesssim}} \sum
\limits _{\xi \in {\bf V}({\cal D}_{\hat t+l,j_l})} w^q(\xi)
\left(\sum \limits _{\xi_*\le \xi'\le \xi} u(\xi')
f(\xi')\right)^q \stackrel{(\ref{ufxi_sum})}
{\underset{\mathfrak{Z}_0}{\lesssim}}
$$
$$
\lesssim \sum \limits _{\xi \in {\bf V}({\cal D}_{\hat t+l,j_l})}
w^q(\xi) \left(\sum \limits _{s=0}^{l-1}  \sum \limits _{\tilde
\xi_{\hat t+s,j_s}\le \xi'<\tilde \xi_{\hat t+s+1,j_{s+1}}}
u(\xi') \|f\| _{X_p(\hat F(\xi'))} \right)^q +
$$
$$
+\sum \limits _{\xi \in {\bf V}({\cal D}_{\hat t+l,j_l})} w^q(\xi)
\left(\sum \limits _{\tilde \xi_{\hat t+l,j_l}\le \xi'\le \xi}
u(\xi') \|f\|_{X_p(\hat F(\xi'))}\right)^q
\stackrel{(\ref{spq_duw1}),
(\ref{v_of_gamma_t})}{\underset{\mathfrak{Z}_0}{\lesssim}}
$$
$$
\lesssim 2^{-q\kappa \cdot 2^{\hat t+l-1}} \cdot 2^{-q\tilde
\alpha_w(\hat t+l)}\left(\sum \limits _{s=0}^{l-1} \sum \limits
_{\tilde \xi_{\hat t+s,j_s}\le \xi'<\tilde \xi_{\hat
t+s+1,j_{s+1}}} u(\xi') \|f\| _{X_p(\hat F(\xi'))} \right)^q
+$$$$+ 2^{\left(1-\frac qp\right)(\hat t+l)} \|f\|^q_{X_p({\cal
D}_{\hat t+l,j_l})};
$$
i.e.,
\begin{align}
\label{fmpf} \begin{array}{c} \|f-P_{\xi_*}f\| _{Y_q(\Omega_{{\cal
D}_{\hat t+l,j_l}})}^q \underset{\mathfrak{Z}_0}{\lesssim}
2^{\left(1-\frac qp\right)(\hat t+l)} \|f\|^q_{X_p({\cal D}_{\hat
t+l,j_l})}+\\ + 2^{-q\kappa \cdot 2^{\hat t+l-1}} \cdot
2^{-q\tilde \alpha_w(\hat t+l)}\left(\sum \limits _{s=0}^{l-1}
\sum \limits _{\tilde \xi_{\hat t+s,j_s}\le \xi'<\tilde \xi_{\hat
t+s+1,j_{s+1}}} u(\xi') \|f\| _{X_p(\hat F(\xi'))} \right)^q.
\end{array}
\end{align}

Denote
\begin{align}
\label{zeta_ti} \{\zeta_{t,i}:\, i\in I'_t\}=\{\zeta\in {\bf
V}_{\max}({\cal D}\cap \Gamma_t):\; {\bf V}_1^{{\cal D}}(\zeta)
\ne \varnothing\}.
\end{align}
By (\ref{v_of_gamma_t}),
\begin{align}
\label{zti_eta} \forall i\in I'_t\;\; \exists s\in \tilde
I_{\lceil 2^t/m_*\rceil-1}:\quad \zeta_{t,i} = \eta_{\lceil
2^t/m_*\rceil-1,s}.
\end{align}

Let us define the tree $\hat {\cal D}$ with vertex set $\{\xi_*\} \cup
\left(\cup _{t\ge \hat t} \{\zeta_{t,i}:\; i\in I'_t\}\right)$.
The partial order on ${\bf V}(\hat {\cal D})$ is defined as
follows. We set
$${\bf V}_1^{\hat{\cal D}}(\zeta_{t,i}) =\{\zeta_{t+1,j}:\;
\zeta_{t+1,j}\in {\cal D}_{\zeta_{t,i}}\}.$$ If $\xi_* \notin
\{\zeta_{\hat t,i}:\; i\in I'_{\hat t}\}$, then we set ${\bf
V}_1^{\hat{\cal D}}(\xi_*)=\{\zeta_{\hat t,i}:\; i\in I'_{\hat
t}\}$.

Further, we define the functions $\varphi$, $\hat u$, $\hat w:{\bf
V}(\hat{\cal D}) \rightarrow \R_+$. We set
\begin{align}
\label{phi_zti} \varphi(\zeta_{t,i}) =\sum _{\xi'\in {\bf
V}(\Gamma_t)\cap {\bf V}({\cal D}), \, \xi'\le \zeta_{t,i}}
\frac{u(\xi')}{u(\zeta_{t,i})} \|f\|_{X_p(\hat F(\xi'))},
\end{align}
\begin{align}
\label{wzti} \hat w(\zeta_{t,i}) = 2^{-\kappa \cdot 2^t}\cdot
2^{-\tilde \alpha_w t}, \quad \hat u(\zeta_{t,i}) =2^{\kappa \cdot
2^t} \cdot 2^{-\alpha_u t} \stackrel{(\ref{uetaji}),
(\ref{zti_eta})} {\underset{\mathfrak{Z}_0}{\asymp}}
u(\zeta_{t,i}).
\end{align}
If $\xi_* \notin \{\zeta_{\hat t,i}:\; i\in I'_{\hat t}\}$, then
we set $\varphi(\xi_*) =\hat w(\xi_*)=\hat u(\xi_*)=0$.

Let $t\ge \hat t$, $j\in J_{t,{\cal D}}$. If $t>\hat t$, then
$\hat \xi_{t,j}\in {\bf V}({\cal D})$. We
take $i\in I'_{t-1}$ such that $\hat \xi_{t,j}\in {\bf V}_1^{{\cal
D}}(\zeta_{t-1,i})$. Then
\begin{align}
\label{est_fpf_tj} \|f-P_{\xi_*}f\|^q_{Y_q({\cal D}_{t,j})}
\stackrel{(\ref{fmpf}), (\ref{wzti})}{\underset{\mathfrak{Z}_0}
{\lesssim}} 2^{\left(1-\frac qp\right) t} \|f\|^q_{X_p({\cal
D}_{t,j})} + \hat w^q(\zeta_{t-1,i}) \left( \sum \limits
_{\zeta'\in {\bf V}(\hat{\cal D}), \, \zeta'\le \zeta_{t-1,i}}
\hat u(\zeta')\varphi(\zeta')\right)^q.
\end{align}
Notice that ${\rm card}\, {\bf V}_1^{{\cal D}} (\zeta_{t,i})
\stackrel{(\ref{card_jj})}{\underset{\mathfrak{Z}_0}{\lesssim}}
1$. Summing (\ref{est_fpf_tj}) over all $t\ge \hat t$, $j\in
J_{t,{\cal D}}$, we obtain
\begin{align}
\label{est_fpf_d}
\begin{array}{c}
\|f-P_{\xi_*}f\|_{Y_q(\Omega_{{\cal D}})}^q
\underset{\mathfrak{Z}_0} {\lesssim} \sum \limits _{t=\hat
t}^\infty 2^{\left(1-\frac qp\right) t}\sum \limits _{j\in
J_{t,{\cal D}}} \|f\|^q_{X_p(\Omega_{{\cal D}_{t,j}})}+
\\
+\sum \limits _{t=\hat t}^{\infty} \sum \limits _{i\in I'_t} \hat
w^q(\zeta_{t,i}) \left( \sum \limits _{\zeta'\in {\bf V}(\hat{\cal
D}), \, \zeta'\le \zeta_{t,i}} \hat
u(\zeta')\varphi(\zeta')\right)^q.
\end{array}
\end{align}

Let us estimate the second summand (we denote it by $L$).

Let $\hat u(\zeta)=\hat u_1(\zeta)\hat u_2(\zeta)$, $\hat
u_1(\zeta_{t,i}) =2^{\varepsilon t}$, where $\varepsilon
=\varepsilon(\mathfrak{Z}_0)>0$ is sufficiently small (it will be
chosen later); if $\xi_* \notin \{\zeta_{\hat t,i}:\; i\in
I'_{\hat t}\}$, then we set $\hat u_1(\xi_*)=\hat u_2(\xi_*)=0$.
By H\"{o}lder's inequality,
$$
L=\sum \limits _{\zeta \in {\bf V}(\hat {\cal D})} \hat w^q(\zeta)
\left(\sum \limits _{\zeta'\le \zeta} \hat u_1(\zeta') \hat
u_2(\zeta') \varphi(\zeta')\right)^q\le
$$
$$
\le \sum \limits _{\zeta \in {\bf V}(\hat {\cal D})} \hat
w^q(\zeta) \left(\sum \limits _{\zeta'\le \zeta} \hat
u_1^{q'}(\zeta')\right)^{\frac{q}{q'}} \sum \limits_{\zeta' \le
\zeta} \hat u_2^q(\zeta') \varphi^q(\zeta')
\underset{\mathfrak{Z}_0}{\lesssim}
$$
$$
\lesssim \sum \limits _{\zeta \in {\bf V}(\hat {\cal D})} \hat
w^q(\zeta) \hat u_1^q(\zeta)\sum \limits_{\zeta' \le \zeta} \hat
u_2^q(\zeta') \varphi^q(\zeta')= \sum \limits _{\zeta' \in {\bf
V}(\hat{\cal D})} \hat u_2^q(\zeta') \varphi^q(\zeta') \sum
\limits _{\zeta \ge \zeta'} \hat w^q(\zeta) \hat u_1^q(\zeta)=:M.
$$
By (\ref{zti_eta}), there exists $s\in \tilde I_{\lceil
2^t/m_*\rceil-1}$ such that
$$
{\rm card}\, {\bf V}_l^{\hat{\cal D}}(\zeta_{t,i}) ={\rm card}\,
\{i'\in I'_{t+l}:\; \zeta_{t+l,i'}\ge \zeta_{t,i}\}\le
$$
$$
\le{\rm card}\, {\bf V}^{\cal A}_{\lceil 2^{t+l}/m_*\rceil-\lceil
2^t/m_*\rceil}(\eta_{\lceil 2^t/m_*\rceil-1,s})
\stackrel{(\ref{card_jj})}{\underset{\mathfrak{Z}_0}{\lesssim}}
\frac{2^{\theta \cdot 2^{t+l}}2^{-\gamma(t+l)}
\tau^{-1}(2^{t+l})}{2^{\theta \cdot 2^t}2^{-\gamma
t}\tau^{-1}(2^t)}.
$$
This together with relations $\kappa =\frac{\theta}{q}$,
$\alpha_w> \frac{1-\gamma}{q}$ (see condition 1 of Assumption
\ref{supp3}) yields that for sufficiently smalll $\varepsilon>0$
$$
\sum \limits _{\zeta\ge \zeta_{t,i}} \hat w^q(\zeta) \hat
u_1^q(\zeta)
\stackrel{(\ref{wzti})}{\underset{\mathfrak{Z}_0}{\lesssim}} \sum
\limits _{l\ge 0} 2^{-q\kappa \cdot 2^{t+l}}\cdot 2^{-q\tilde
\alpha_w (t+l)} \cdot 2^{q\varepsilon (t+l)}\frac{2^{\theta \cdot
2^{t+l}}2^{-\gamma(t+l)}\tau^{-1}(2^{t+l})}{2^{\theta \cdot
2^t}2^{-\gamma t}\tau^{-1}(2^{t})}=
$$
$$
=2^{-\theta \cdot2^t}\cdot 2^{\gamma t} \sum \limits _{l\ge
0}2^{-q\left(\tilde \alpha_w+\frac{\gamma}{q}-\varepsilon\right)
(t+l)}\frac{\tau(2^t)}{\tau(2^{t+l})}
\stackrel{(\ref{te})}{\underset{\mathfrak{Z}_0}{\lesssim}}
2^{-q\kappa \cdot 2^t} \cdot
2^{-q(\tilde\alpha_w-\varepsilon) t} \stackrel{(\ref{wzti})}{=}
\hat w^q(\zeta_{t,i}) \hat u_1^q(\zeta_{t,i}).
$$
Thus,
$$
M \underset{\mathfrak{Z}_0}{\lesssim} \sum \limits _{\zeta' \in
{\bf V}(\hat{\cal D})} \hat w^q(\zeta')\hat u_1^q(\zeta')\hat
u_2^q(\zeta') \varphi^q(\zeta')=\sum \limits _{\zeta' \in {\bf
V}(\hat{\cal D})} \hat w^q(\zeta')\hat
u^q(\zeta')\varphi^q(\zeta') \stackrel{(\ref{auaw}),
(\ref{wzti})}{=}
$$
$$
=\sum \limits _{t=\hat t}^\infty \sum \limits _{j\in J_{t,{\cal
D}}} 2^{\left(1-\frac qp\right)t} \sum \limits _{i\in I'_t:\;
\zeta _{t,i} \in {\bf V}({\cal D}_{t,j})} \varphi^q(\zeta_{t,i});
$$
i.e.,
\begin{align}
\label{est_sum} \sum \limits _{t=\hat t}^{\infty} \sum \limits
_{i\in I'_t} \hat w^q(\zeta_{t,i}) \left( \sum \limits _{\zeta'\in
{\bf V}(\hat{\cal D}), \, \zeta'\le \zeta_{t,i}} \hat
u(\zeta')\varphi(\zeta')\right)^q
\underset{\mathfrak{Z}_0}{\lesssim} \sum \limits _{t=\hat
t}^\infty \sum \limits _{j\in  J_{t,{\cal D}}} 2^{\left(1-\frac
qp\right)t} \sum \limits _{i\in I'_t:\; \zeta _{t,i} \in {\bf
V}({\cal D}_{t,j})} \varphi^q(\zeta_{t,i}).
\end{align}

It remains to prove that for any $t\ge \hat t$, $j\in J_{t,{\cal
D}}$
\begin{align}
\label{est1} \sum \limits _{i\in I'_t:\; \zeta _{t,i} \in {\bf
V}({\cal D}_{t,j})} \varphi^q(\zeta_{t,i})
\underset{\mathfrak{Z}_0}{\lesssim} \|f\|^q_{X_p(\Omega _{{\cal
D}_{t,j}})}.
\end{align}
Then (\ref{est_fpf_d}), (\ref{est_sum}), (\ref{est1}) imply
(\ref{fpxifd}).

By (\ref{txi_def}) and (\ref{phi_zti}),
$$
\sum \limits _{i\in I'_t:\; \zeta _{t,i} \in {\bf V}({\cal
D}_{t,j})} \varphi^q(\zeta_{t,i}) = \sum \limits _{i\in I'_t:\;
\zeta _{t,i} \in {\bf V}({\cal D}_{t,j})} \left(\sum \limits
_{\tilde \xi_{t,j}\le \xi'\le \zeta_{t,i}} \frac{u(\xi')}
{u(\zeta_{t,i})} \|f\|_{X_p(\hat F(\xi'))}\right)^q=
$$
$$
=\sum \limits _{i\in I'_t:\; \zeta _{t,i} \in {\bf V}({\cal
D}_{t,j})} u^{-q}(\zeta_{t,i})\left(\sum \limits _{\tilde
\xi_{t,j}\le \xi'\le \zeta_{t,i}} u(\xi') \|f\|_{X_p(\hat
F(\xi'))}\right)^q=:S.
$$

From (\ref{zeta_ti}) it follows that $\zeta_{t,i}\in {\bf
V}_{\max}({\cal A}_{t,j})$.

We denote ${\cal A}'_{t,j}=({\cal A}_{t,j})_{\tilde \xi_{t,j}}$.
Then ${\bf V}({\cal D}_{t,j}) \subset {\bf V}({\cal A}'_{t,j})$.
Given $\sigma\ge 0$, we set
$$
w_{(\sigma)}(\xi) = \left\{
\begin{array}{l} u^{-1}(\xi) \quad \text{for}\quad \xi \in {\bf
V}_{\max}({\cal A}'_{t,j}),
\\ \sigma \quad \text{for}\quad \xi \in {\bf V}({\cal A}'_{t,j}) \backslash
{\bf V}_{\max}({\cal A}'_{t,j}).\end{array} \right.
$$
Then
$$
S\le \left[\mathfrak{S}^{p,q}_{{\cal A}'_{t,j}, u,
w_{(\sigma)}}\right]^q \|f\|_{X_p(\Omega_{{\cal D}_{t,j}})}.
$$
We claim that for sufficiently small $\sigma>0$ the estimate
$\mathfrak{S}^{p,q}_{{\cal A}'_{t,j}, u, w_{(\sigma)}}
\underset{\mathfrak{Z}_0}{\lesssim} 1$ holds. This implies
(\ref{est1}).

Let $\tilde \xi_{t,j} \in {\bf V}^{{\cal
A}}_{s_*-j_{\min}}(\eta_{j_{\min},1})$. From (\ref{v_of_gamma_t})
and (\ref{txi_def}) it follows that
\begin{align}
\label{s_star} 2^{t-1}\le m_*s_*<2^t.
\end{align}
Moreover, ${\cal A}'_{t,j}=[{\cal A}_{\tilde \xi_{t,j}}]_{\le
\lceil 2^t/m_*\rceil-1-s_*}$ (since $\hat\xi_{t+1,i}\in {\bf
V}^{\cal A}_{\lceil 2^t/m_*\rceil-j_{\min}}(\eta_{j_{\min},1})$,
$i\in \hat J_t$). By condition 5b of Assumption \ref{supp3}, there
exists a vertex $\hat \zeta \in {\bf V}^{\hat{\cal
A}}_{s_*-j_{\min}}(\hat \zeta_0)$ such that for the tree $\hat
{\cal A}_{t,j}:=[\hat{\cal A}_{\hat \zeta}]_{\le \lceil
2^t/m_*\rceil-1-s_*}$ and for the functions $\overline{u}$,
$\overline{w}_{(\sigma)}:{\bf V}(\hat {\cal A}_{t,j}) \rightarrow
(0, \, \infty)$ defined by
\begin{align}
\label{w_sigma} \overline{u}(\xi) =u_s, \quad \xi \in {\bf
V}^{\hat{\cal A}}_{s-s_*}(\hat \zeta_0), \quad
\overline{w}(\xi)=\left\{
\begin{array}{l} u^{-1}_{\lceil 2^t/m_*\rceil-1}, \quad \xi \in {\bf V}^{\hat{\cal
A}}_{\lceil 2^t/m_*\rceil-1-s_*}(\hat \zeta_0), \\ \sigma, \quad
\xi \in {\bf V}^{\hat{\cal A}}_{s-s_*}(\hat \zeta_0), \quad
s<\lceil 2^t/m_*\rceil-1,\end{array}\right.
\end{align}
the inequality $\mathfrak{S}^{p,q}_{{\cal A}'_{t,j}, u,
w_{(\sigma)}} \underset{\mathfrak{Z}_0}{\lesssim}
\mathfrak{S}^{p,q}_{\hat{\cal A}_{t,j}, \overline{u},
\overline{w}_{(\sigma)}}$ holds.

From (\ref{card_jj1}) it follows that we can apply
Theorem 3.6 in \cite{vas_har_tree} and estimate
$\mathfrak{S}^{p,q}_{\hat{\cal A}_{t,j}, \overline{u},
\overline{w}_{(\sigma)}}$ from above.
For sufficiently small $\sigma=\sigma(\mathfrak{Z}_0,
\, t)>0$ we get
$$
\mathfrak{S}^{p,q}_{\hat{\cal A}_{t,j}, \overline{u},
\overline{w}_{(\sigma)}} \stackrel{(\ref{card_jj1}),
(\ref{w_sigma})}{\underset{\mathfrak{Z}_0}{\lesssim}}
 \sup _{s_*\le s< \lceil 2^t/m_*\rceil} \left(\sum \limits
_{l=s_*}^{s} u_l^{p'}\right)^{1/p'} \times$$$$\times \left(\sum
\limits _{l=s}^{\lceil 2^t/m_*\rceil-2} \sigma^q \cdot 2^{\theta
m_*(l-s)}\frac{s^\gamma \tau(m_*s)}{l^\gamma \tau(m_*l)}+
u_{\lceil2^t/m_*\rceil-1}^{-q} \cdot
2^{\theta(2^t-m_*s)}\frac{(m_*s)^\gamma \tau(m_*s)}{2^{\gamma t}
\tau(2^t)}\right)^{1/q}
\stackrel{(\ref{uetaji}),(\ref{s_star})}{\underset{\mathfrak{Z}_0}{\lesssim}}
$$
$$
\lesssim \sup _{s_*\le s\le \lceil 2^t/m_*\rceil-1} \left(\sum
\limits _{l=s_*}^{s} 2^{p'\kappa m_*l} \cdot (m_*l)^{-\alpha_u
p'}\right)^{1/p'} \left(2^{-q\kappa \cdot 2^t} \cdot
2^{q\alpha_ut} \cdot 2^{\theta(2^t-m_*s)}\right)^{1/q}.
$$
Since $\theta=q\kappa$, we have
$$
\mathfrak{S}^{p,q}_{\hat{\cal A}_{t,j}, \overline{u},
\overline{w}_{(\sigma)}} \underset{\mathfrak{Z}_0}{\lesssim} \sup
_{s_*\le s\le \lceil 2^t/m_*\rceil-1} 2^{\kappa m_*s} \cdot
(m_*s)^{-\alpha_u} \cdot 2^{-\kappa m_*s} \cdot 2^{\alpha_u t}
\stackrel{(\ref{s_star})}{\underset{\mathfrak{Z}_0}{\lesssim}} 1.
$$
This completes the proof.
\end{proof}

Thus, Assumptions \ref{sup1}, \ref{sup2}, \ref{sup3} hold, and by
Theorem \ref{main_tree} we get (\ref{main_est}).

Now we suppose that instead of Assumption \ref{supp3} the
following condition holds.
\begin{Supp}
\label{supp4} There exist numbers $m_*\in \N$, $\kappa>0$,
$\gamma\le 0$, $\nu\in \R$, $\alpha_*\in \R$, $\lambda_u\in \R$,
$\lambda_w\in \R$ such that $\lambda_u+\lambda_w=\frac 1p-\frac
1q$ and the following assertions hold.
\begin{enumerate}
\item For any $j'\ge j\ge j_{\min}$ and for any vertex $\xi \in {\bf V}
_{j-j_{\min}}^{{\cal A}}(\xi_0)$
\begin{align}
\label{card_jj1_1} {\rm card}\, {\bf V}_{j'-j}^{{\cal A}}(\xi) \le
c_3 \frac{(m_*j)^\gamma |\log (m_*j)|^\nu}{(m_*j')^\gamma|\log
(m_*j')|^\nu}.
\end{align}
\item For any $j\ge j_{\min}$, $i\in \tilde I_j$
\begin{align}
\label{uetaji1} \begin{array}{c} u(\eta_{j,i}) =u_j= 2^{\kappa
m_*j}(m_*j)^{\alpha_*}|\log (m_*j)|^{-\lambda_u}, \\
w(\eta_{j,i}) =w_j=2^{-\kappa m_*j} (m_*j)^{-\alpha_*}|\log
(m_*j)|^{-\lambda_w}. \end{array}
\end{align}
\end{enumerate}
\end{Supp}

The partition $\{{\cal A}_{t,i}\}_{t\ge t_0, \, i\in \hat J_t}$ is
defined as follows. Let $$t_0=\min \{t\in \Z_+:\; 2^{2^t}> m_*
j_{\min}\}.$$ Given $t\ge t_0$, we denote by $\Gamma_t$ the
maximal subgraph in ${\cal A}$ on the vertex set $\{\eta_{j,s}:\;
2^{2^{t-1}}\le m_*j<2^{2^t}, \; s\in \tilde I_j\}$, and by ${\cal
A}_{t,i}$, $i\in \hat J_t$, the connected components of
$\Gamma_t$. Then
\begin{align}
\label{c_vgt_1} {\rm card}\, {\bf V}(\Gamma_t)
\underset{\mathfrak{Z}_0}{\lesssim} 2^{(1-\gamma) 2^t}2^{-\nu t}.
\end{align}

Repeating the proof of Lemma 5.1 in \cite{vas_vl_raspr2} and
taking into account that $p<q$, we get the following assertion.

\begin{Lem}
\label{lem_est1} Let ${\cal D}$ be a subtree in ${\cal A}$, and
let $\xi_*$ be its minimal vertex. Then
\begin{align}
\label{fpxifd1} \|f-P_{\xi_*}f\|_{Y_q({\cal D})}^q
\underset{\mathfrak{Z}_0}{\lesssim} \sum \limits _{t=t_0}
^{\infty} 2^{(1-\frac qp)t} \sum \limits _{i\in J_{t, {\cal D}}}
\|f\|_{X_p(\Omega _{{\cal D}_{t,i}})} ^q.
\end{align}
\end{Lem}

From Assumption \ref{supp2}, (\ref{c_vgt_1}) and (\ref{fpxifd1})
we obtain Assumptions \ref{sup1}, \ref{sup2}, \ref{sup3}.

\begin{Rem}
\label{discr_rem} Suppose that Assumption \ref{supp2} is replaced
by the following condition: for any $\xi\in {\bf V}({\cal A})$ the
set $\hat F(\xi)$ is the atom of ${\rm mes}$. Then the assertion
of Theorem \ref{main_tree} holds as well (see Remark
\ref{discr_rem0}).
\end{Rem}

\section{Estimates for entropy numbers of embeddings of weighted Sobolev spaces}

Let us define the weighted Sobolev class $W^r_{p,g}(\Omega)$ and
the weighted Lebesgue space $L_{q,v}(\Omega)$.

Let $\Omega \subset \R^d$ be a bounded domain, and let $g$,
$v:\Omega\rightarrow (0, \, \infty)$ be measurable functions. For
each measurable vector-valued function $\psi:\ \Omega\rightarrow
\R^l$, $\psi=(\psi_k) _{1\le k\le l}$, and for each $p\in [1, \,
\infty)$, we put
$$
\|\psi\|_{L_p(\Omega)}= \Big\|\max _{1\le k\le l}|\psi _k |
\Big\|_p=\left(\int \limits_\Omega \max _{1\le k\le l}|\psi _k(x)
|^p\, dx\right)^{1/p}.
$$
Let $\overline{\beta}=(\beta _1, \, \dots, \, \beta _d)\in
\Z_+^d:=(\N\cup\{0\})^d$, $|\overline{\beta}| =\beta _1+
\ldots+\beta _d$. For any distribution $f$ defined on $\Omega$ we
write $\displaystyle \nabla ^r\!f=\left(\partial^{r}\! f/\partial
x^{\overline{\beta}}\right)_{|\overline{\beta}| =r}$ (here partial
derivatives are taken in the sense of distributions), and denote
by $l_{r,d}$ the number of components of the vector-valued
distribution $\nabla ^r\!f$. We set
$$
W^r_{p,g}(\Omega)=\left\{f:\ \Omega\rightarrow \R\big| \; \exists
\psi :\ \Omega\rightarrow \R^{l_{r,d}}\!:\ \| \psi \|
_{L_p(\Omega)}\le 1, \, \nabla ^r\! f=g\cdot \psi\right\}
$$
\Big(we denote the corresponding function $\psi$ by
$\displaystyle\frac{\nabla ^r\!f}{g}$\Big),
$$
\| f\|_{L_{q,v}(\Omega)}{=}\| f\|_{q,v}{=}\|
fv\|_{L_q(\Omega)},\qquad L_{q,v}(\Omega)=\left\{f:\Omega
\rightarrow \R| \; \ \| f\| _{q,v}<\infty\right\}.
$$
We call the set $W^r_{p,g}(\Omega)$ a weighted Sobolev class.
Observe that if $g\in L_{p'}^{{\rm loc}}(\Omega)$, then $\nabla^r
f\in L_1^{{\rm loc}}(\Omega)$.

For $x\in \R^d$ and $a>0$ we shall denote by  $B_a(x)$ the closed
Euclidean ball of radius $a$ in $\R^d$ centered at the point $x$.
\begin{Def}
\label{fca} Let $\Omega\subset\R^d$ be a bounded domain, and let
$a>0$. We say that $\Omega \in {\bf FC}(a)$ if there exists a
point $x_*\in \Omega$ such that, for any $x\in \Omega$, there
exist a number $T(x)>0$ and a curve $\gamma _x:[0, \, T(x)]
\rightarrow\Omega$ with the following properties:
\begin{enumerate}
\item $\gamma _x\in AC[0, \, T(x)]$, $\left|\frac{d \gamma _x(t)}{dt}\right|=1$ a.e.,
\item $\gamma _x(0)=x$, $\gamma _x(T(x))=x_*$,
\item $B_{at}(\gamma _x(t))\subset \Omega$ for any $t\in [0, \, T(x)]$.
\end{enumerate}
\end{Def}

\smallskip

\begin{Def}
We say that $\Omega$ satisfies the John condition (and call
$\Omega$ a John domain) if $\Omega\in {\bf FC}(a)$ for some $a>0$.
\end{Def}

For a bounded domain the John condition is equivalent to the
flexible cone condition (see the definition in \cite{besov_il1}).
As examples of such domains we can take \begin{enumerate}
\item domains with Lipschitz boundary; \item the  Koch's
snowflake; \item domains $\Omega=\cup _{0< t\le T} {\rm int}\,
B_{ct}(\gamma(t))$, where $\gamma:[0, \, T] \rightarrow \R^d$ is a
curve with natural parametrization and $c>0$.
\end{enumerate} Domains with zero inner angles do not satisfy the
John condition.

We denote by $\mathbb{H}$ the set of all nondecreasing positive
functions defined on $(0, \, 1]$.
\begin{Def}
\label{h_set} {\rm (see \cite{m_bricchi1}).} Let $\Gamma\subset
\R^d$ be a nonempty compact set and $h\in \mathbb{H}$. We say that
$\Gamma$ is an $h$-set if there are a constant $c_*\ge 1$ and a
finite countably additive measure $\mu$ on $\R^d$ such that $\supp
\mu=\Gamma$ and
\begin{align}
\label{c1htmu} c_*^{-1}h(t)\le \mu(B_t(x))\le c_* h(t)
\end{align}
for any $x\in \Gamma$ and $t\in (0, \, 1]$.
\end{Def}

\begin{Exa}
Let $\Gamma \subset \R^d$ be a Lipschitz manifold of dimension
$k$, $0\le k<d$. Then $\Gamma$ is an $h$-set with $h(t)=t^k$.
\end{Exa}

\begin{Exa}
Let $\Gamma \subset \R^2$ be the Koch curve. Then $\Gamma$ is an
$h$-set with $h(t)=t^{\log 4/\log 3}$ (see \cite[p.
66--68]{p_mattila}).
\end{Exa}

Let $\Omega\in {\bf FC}(a)$ be a bounded domain, and let
$\Gamma\subset \partial \Omega$ be an $h$-set. Below we consider a
function $h\in \mathbb{H}$ which has the following form near zero:
\begin{align}
\label{def_h} h(t)=t^{\theta}|\log t|^{\gamma}\tau(|\log t|),
\;\;\; 0\le \theta<d,
\end{align}
where $\tau:(0, \, +\infty)\rightarrow (0, \, +\infty)$ is an
absolutely continuous function such that
\begin{align}
\label{yty} \frac{t\tau'(t)}{\tau(t)} \underset{t\to+\infty}{\to}
0.
\end{align}

Let $1<p\le \infty$, $1\le q<\infty$, $r\in \N$, $\delta:=r+\frac
dq-\frac dp>0$, $\beta_g$, $\beta_v\in \R$, $g(x)=\varphi_g({\rm
dist}_{|\cdot|}(x, \, \Gamma))$, $v(x)=\varphi_v({\rm
dist}_{|\cdot|}(x, \, \Gamma))$,
\begin{align}
\label{ghi_g0} \varphi_g(t)=t^{-\beta_g}|\log
t|^{-\alpha_g}\rho_g(|\log t|), \;\; \varphi_v(t)=t^{-\beta_v}
|\log t|^{-\alpha_v}\rho_v(|\log t|),
\end{align}
where $\rho_g$ and $\rho_v$ are absolutely continuous functions
such that
\begin{align}
\label{psi_cond} \frac{t\rho'_g(t)}{\rho_g(t)}
\underset{t\to+\infty}{\to}0, \;\; \frac{t\rho'_v(t)}{\rho_v(t)}
\underset{t\to+\infty}{\to}0.
\end{align}

Moreover, assume that
\begin{align}
\label{muck} \beta_v<\frac{d-\theta}{q} \quad \text{or} \quad
\beta_v=\frac{d-\theta}{q}, \quad \alpha_v>\frac{1-\gamma}{q}.
\end{align}

Without loss of generality we may consider
$\overline{\Omega}\subset \left(-\frac 12, \, \frac 12\right)^d$.

Denote $$\beta=\beta_g+\beta_v,\quad \alpha=\alpha_g+\alpha_v,
\quad \rho(y)=\rho_g(y)\rho_v(y),$$ $\mathfrak{Z}=(r,\, d, \, p,
\, q, \, g, \, v, \, h, \, a, \, c_*)$,
$\mathfrak{Z}_*=(\mathfrak{Z}, \, R)$, where $c_*$ is the constant
from Definition \ref{h_set} and \label{r_def}$R={\rm diam}\,
\Omega$.

Denote by ${\cal P}_{r-1}(\R^d)$ the space of polynomials on
$\R^d$ of degree not exceeding $r-1$. For a measurable set
$E\subset \R^d$ we put $${\cal P}_{r-1}(E)= \{f|_E:\, f\in {\cal
P}_{r-1}(\R^d)\}.$$ Observe that $W^r_{p,g}(\Omega) \supset {\cal
P}_{r-1}(\Omega)$.

In Theorems \ref{th1}, \ref{th3} the conditions on weights are
such that $W^r_{p,g}(\Omega) \subset L_{q,v}(\Omega)$ and there
exist $M>0$ and a linear continuous operator $P:L_{q,v}(\Omega)
\rightarrow {\cal P}_{r-1}(\Omega)$ such that for any function
$f\in W^r_{p,g}(\Omega)$
\begin{align}
\label{fpflqv} \|f-Pf\|_{L_{q,v}(\Omega)} \le M
\left\|\frac{\nabla^r f}{g}\right\| _{L_p(\Omega)}
\end{align}
(see \cite{vas_vl_raspr, vas_vl_raspr2, vas_sib, vas_w_lim}).

Denote ${\cal W}^r_{p,g}(\Omega) ={\rm span}\, W^r_{p,g}(\Omega)$,
$\hat {\cal W}^r_{p,g}(\Omega) =\{f-Pf:\; f\in {\cal
W}^r_{p,g}(\Omega)\}$. Let $\hat {\cal W}^r_{p,g}(\Omega)$ be
equipped with norm $\|f\|_{\hat {\cal W}^r_{p,g}(\Omega)}:=
\left\|\frac{\nabla^r f}{g}\right\| _{L_p(\Omega)}$. Denote by
$I:\hat {\cal W}^r_{p,g}(\Omega) \rightarrow L_{q,v}(\Omega)$ the
embedding operator. From (\ref{fpflqv}) it follows that it is
continuous.

First we consider the case $0<\theta<d$. We set
$$
\alpha_0:=\left\{ \begin{array}{l} \alpha \quad \text{for} \quad
\beta_v<\frac{d-\theta}{q}, \\ \alpha-\frac 1q \quad \text{for}
\quad \beta_v=\frac{d-\theta}{q}.\end{array} \right.
$$
In \cite{vas_entr} the estimates for entropy numbers $e_n(I:\hat
{\cal W}^r_{p,g}(\Omega) \rightarrow L_{q,v}(\Omega))$ were
obtained under the following conditions. In the case
$\delta-\beta>\theta\left(\frac 1q-\frac 1p\right)_+$ we assumed
that $\frac{\delta}{d}\ne \frac{\delta-\beta}{\theta}$. In the
case $\delta-\beta=\theta\left(\frac 1q-\frac 1p\right)_+$ we
assumed that $\alpha_0\ne \frac 1p-\frac 1q$ for $p<q$. Now we
obtain estimates for $p<q$, $\beta-\delta=0$, $\alpha_0= \frac
1p-\frac 1q$.

\begin{Trm} \label{th1}
Suppose that the conditions (\ref{def_h})--%, (\ref{yty}),
%(\ref{ghi_g0}),
(\ref{muck}) hold, $\rho_g\equiv 1$, $\rho_v\equiv 1$. Let
$0<\theta<d$, $p<q$, $\beta=\delta$, and let $\alpha_0= \frac
1p-\frac 1q$. Then
$$
e_n(I:\hat {\cal W}^r_{p,g}(\Omega) \rightarrow L_{q,v}(\Omega))
\underset{\mathfrak{Z}_*}{\asymp} n^{\frac 1q-\frac 1p}.
$$
\end{Trm}

\smallskip

Now we consider the case $\theta=0$, $\beta_v<\frac{d}{q}$. We
assume that
\begin{align}
\label{rho_g_l} \rho_g(t)=|\log t|^{-\lambda_g}, \quad
\rho_v=|\log t|^{-\lambda_v}, \quad \tau(t)=|\log t|^{\nu}.
\end{align}
Denote $\lambda =\lambda_g+\lambda_v$.

In \cite{vas_entr} estimates of entropy numbers were obtained in
the following cases:
\begin{enumerate}
\item $\delta-\beta>\theta\left(\frac 1q-\frac 1p\right)_+$,
\item $\delta-\beta=\theta\left(\frac 1q-\frac 1p\right)_+$,
$\alpha>(1-\gamma) \left(\frac 1q-\frac 1p\right)_+$,
$\frac{\delta}{d}\ne \frac{\alpha}{1-\gamma}$,
\item $\delta-\beta=\theta\left(\frac 1q-\frac 1p\right)_+$,
$\alpha=(1-\gamma) \left(\frac 1q-\frac 1p\right)_+$, and if
$p<q$, then $\lambda \ne \frac 1p-\frac 1q$.
\end{enumerate}
Here we obtain the estimates for $p<q$, $\delta-\beta=0$,
$\alpha=0$, $\lambda = \frac 1p-\frac 1q$.

\begin{Trm} \label{th3}
Let the conditions (\ref{def_h}), (\ref{ghi_g0}), (\ref{rho_g_l})
hold, and let $\theta=0$, $\beta-\delta=0$, $\beta_v<\frac{d}{q}$,
$\alpha=0$, $p<q$, $\lambda= \frac 1p-\frac 1q$. Then
$$
e_n(I:\hat {\cal W}^r_{p,g}(\Omega) \rightarrow L_{q,v}(\Omega))
\underset{\mathfrak{Z}_*}{\asymp} n^{\frac 1q-\frac 1p}.
$$
\end{Trm}

\renewcommand{\proofname}{\bf Proof of Theorems \ref{th1} and \ref{th3}}
\begin{proof}
The lower estimates are proved similarly as in \cite{vas_entr}.

Let us obtain upper estimates. To this end, we apply Theorem
\ref{main_tree}. Consider the tree ${\cal A}$ with vertex set
$\{\eta_{j,i}\}_{j\ge j_{\min}, i\in \tilde I_j}$ and the
partition of the domain $\Omega$ into subdomains
$\Omega[\eta_{j,i}]$, as defined in \cite{vas_vl_raspr},
\cite{vas_vl_raspr2}. We set $\hat F(\eta_{j,i})
=\Omega[\eta_{j,i}]$. Let the number $\overline{s}=\overline{s}(a,
\, d)\ge 4$ be as defined in \cite{vas_vl_raspr}. Repeating
arguments from the proof of Theorem 2 in \cite{vas_vl_raspr2}
(without summation over vertices $\xi \in {\bf V}({\cal
A}_{\xi_*})$), we obtain that Assumption \ref{supp1} holds with
$u(\eta_{j,i})= \varphi_g(2^{-\overline{s}j}) \cdot 2^{-\left(r -
\frac dp\right)\overline{s}j}$, $w(\eta_{j,i}) =
\varphi_v(2^{-\overline{s}j}) \cdot 2^{-\frac{d
\overline{s}j}{q}}$. Assumption \ref{supp2} holds with
$\delta_*=\frac{\delta}{d}$ (see \cite{vas_john}, \cite[Lemma
8]{vas_width_raspr}, \cite{besov_peak_width}). If $\theta>0$, then
Assumption \ref{supp3} holds; condition 5 of this assumption
follows from Lemma 2 in \cite{vas_sib}. If $\theta=0$, then
Assumption \ref{supp4} holds. In both cases we have $\kappa =\frac
dq-\beta_v$, $\alpha_u=\alpha_g$, $\alpha_w=\alpha_v$. In the case
$\theta=0$ we have $\lambda_u=\lambda_g$, $\lambda_w=\lambda_v$.
By Theorem \ref{main_tree}, we obtain the desired upper estimates
of entropy numbers.
\end{proof}
\renewcommand{\proofname}{\bf Proof}

\section{Estimates for entropy numbers of two-weighted summation operators on a tree}

Applying Remark \ref{discr_rem}, we obtain estimates for entropy
numbers of weighted summation operators on trees.

Let ${\cal G}$ be a graph. Given a function $f:{\bf V}({\cal
G})\rightarrow \R$, we set
\begin{align}
\label{flpg} \|f\|_{l_p({\cal G})}=\left ( \sum \limits _{\xi \in
{\bf V} ({\cal G})}|f(\xi)|^p\right )^{1/p} \quad \text{for} \quad
1\le p<\infty, \quad \|f\|_{l_\infty({\cal G})}=\sup _{\xi \in
{\bf V} ({\cal G})} |f(\xi)|.
\end{align}
Denote by $l_p({\cal G})$ the space of functions $f:{\bf V}({\cal
G})\rightarrow \R$ with finite norm $\|f\|_{l_p({\cal G})}$.

Let ${\cal A}$ be a tree, ${\bf V}({\cal A}) =\{\eta_{j,i}:\; j\in
\Z_+, \;\; i\in I_j\}$, let $\eta_{0,1}$ be the minimal vertex of
${\cal A}$, and let  ${\bf V}^{{\cal A}}_j(\eta_{0,1})
=\{\eta_{j,i}\}_{i\in I_j}$ for any $j\in \Z_+$. Suppose that for
some $c_*\ge 1$, $m_*\in \N$
$$
c_*^{-1}\frac{h(2^{-m_*j})} {h(2^{-m_*(j+l)})}\le {\bf V}^{\cal
A}_l(\eta_{j,i})\le  c_*\frac{h(2^{-m_*j})} {h(2^{-m_*(j+l)})},
\quad j, \; l\in \Z_+,
$$
where the function $h\in \mathbb{H}$ is defined by (\ref{def_h})
near zero. Let $u$, $w:{\bf V}({\cal A}) \rightarrow (0, \,
\infty)$, $u(\eta_{j,i})=u_j$, $w(\eta_{j,i})=w_j$, $j\in \Z_+$,
$i\in I_j$,
\begin{align}
\label{uj} u_j=2^{\kappa m_*j}(m_*j+1)^{-\alpha_u}, \quad
w_j=2^{-\kappa m_*j}(m_*j+1)^{-\alpha_w}.
\end{align}
In addition, we suppose that
\begin{align}
\label{muck_1} \kappa>\frac{\theta}{q} \quad \text{or} \quad
\kappa=\frac{\theta}{q}, \quad \alpha_w>\frac{1-\gamma}{q}.
\end{align}

We set $\mathfrak{Z}=(p, \, q, \, u, \, w, \, h, \, m_*, \, c_*)$.

Denote $\tilde \alpha = \alpha_u+\alpha_w$ for
$\kappa>\frac{\theta}{q}$, $\tilde \alpha= \alpha_u+\alpha_w
-\frac 1q$ for $\kappa=\frac{\theta}{q}$.

In \cite{vas_entr} estimates for $e_n(S_{u,w,{\cal A}}:l_p({\cal
A}) \rightarrow l_q({\cal A}))$ were obtained in the case $\tilde
\alpha \ne \frac 1p-\frac 1q$. The case $\tilde \alpha = \frac
1p-\frac 1q$, $q=\infty$ was considered by Lifshits and Linde
\cite{lifs_m, l_l1}. To this end estimates for entropy numbers of
the dual operator $S^*_{u,w,{\cal A}}:l_1({\cal A}) \rightarrow
l_{p'}({\cal A})$ were obtained and the result of the paper
\cite{pajor_tomczak} was applied.

\begin{Trm}
\label{discr_t1} Let $\theta>0$, $1<p<q<\infty$,
$\tilde\alpha=\frac 1p-\frac 1q$. Then
$$
e_n(S_{u,w,{\cal A}}:l_p({\cal A}) \rightarrow l_q({\cal A}))
\underset{\mathfrak{Z}}{\asymp} n^{\frac 1q-\frac 1p}.
$$
\end{Trm}

Now we suppose that $\theta=0$, $\kappa>0$, and instead of
(\ref{uj}) the following condition holds:
\begin{align}
\label{uj00} u_j=2^{\kappa m_*j}(m_*j+1)^{\alpha} [\log
(m_*j+1)]^{-\lambda_u}, \quad w_j=2^{-\kappa
m_*j}(m_*j+1)^{-\alpha}[\log (m_*j+1)]^{-\lambda_w}
\end{align}
with $\alpha\in \R$.

\begin{Trm}
\label{discr_t2} Let $\theta=0$, $\gamma\le 0$, $1<p<q<\infty$,
$\kappa>0$, $\lambda_u+\lambda_w=\frac 1p-\frac 1q$. Then
$$
e_n(S_{u,w,{\cal A}}:l_p({\cal A}) \rightarrow l_q({\cal A}))
\underset{\mathfrak{Z}}{\asymp} n^{\frac 1q-\frac 1p}.
$$
\end{Trm}


\begin{Biblio}
\bibitem{besov_peak_width} O.V. Besov, ``Kolmogorov widths
of Sobolev classes on an irregular domain'', {\it Proc. Steklov
Inst. Math.}, {\bf 280} (2013), 34-45.

\bibitem{besov_il1} O.V. Besov, V.P. Il'in, S.M. Nikol'skii,
{\it Integral representations of functions, and imbedding
theorems}. ``Nauka'', Moscow, 1996. [Winston, Washington DC;
Wiley, New York, 1979].


\bibitem{birm} M.Sh. Birman and M.Z. Solomyak, ``Piecewise polynomial
approximations of functions of classes $W^\alpha_p$'', {\it Mat.
Sb.} {\bf 73}:3 (1967), 331-–355.

\bibitem{pajor_tomczak} J. Bourgain, A. Pajor, S. Szarek, N.
Tomczak-Jaegermann, ``On the duality problem for entropy numbers
of operators'', {\it Geometric Aspects of Functional Analysis,
Lecture Notes in Mathematics} {\bf 1376}, 50--63.

\bibitem{m_bricchi1} M. Bricchi, ``Existence and properties of
h-sets'', {\it Georgian Mathematical Journal}, {\bf 9}:1 (2002),
13–-32.

\bibitem{carl_steph} B. Carl, I. Stephani, {\it Entropy, Compactness, and
the Approximation of Operators}. Cambridge Tracts in Mathematics,
V. 98. Cambridge: Cambridge University Press, 1990.

\bibitem{edm_netr1} D.E. Edmunds, Yu.V. Netrusov, ``Entropy numbers of operators
acting between vector-valued sequence spaces'', {\it Math.
Nachr.}, {\bf 286}:5--6 (2013), 614--630.

\bibitem{edm_netr2} D.E. Edmunds, Yu.V. Netrusov, ``Sch\"{u}tt's theorem for
vector-valued sequence spaces'', {\it J. Approx. Theory}, {\bf
178} (2014), 13--21.

\bibitem{edm_trieb_book} D.E. Edmunds, H. Triebel, {\it Function spaces,
entropy numbers, differential operators}. Cambridge Tracts in
Mathematics, {\bf 120} (1996). Cambridge University Press.

\bibitem{har94_1} D.D. Haroske, ``Entropy numbers in weighted function spaces and
eigenvalue distributions of some degenerate pseudodifferential
operators. I'', {\it Math. Nachr.}, {\bf 167} (1994), 131--156.

\bibitem{har94_2} D.D. Haroske, ``Entropy numbers in weighted function spaces and
eigenvalue distributions of some degenerate pseudodifferential
operators. II'', {\it Math. Nachr.}, {\bf 168} (1994), 109--137.

\bibitem{har_tr05} D.D. Haroske, H. Triebel, ``Wavelet bases and entropy numbers in
weighted function spaces'', {\it Math. Nachr.}, {\bf 278}:1--2
(2005), 108--132.

\bibitem{haroske} D.D. Haroske, L. Skrzypczak, ``Entropy and approximation
numbers of function spaces with Muckenhoupt weights'', {\it Rev.
Mat. Complut.}, {\bf 21}:1 (2008), 135--177.

\bibitem{haroske2} D.D. Haroske, L. Skrzypczak, ``Entropy and approximation numbers of embeddings of
function spaces with Muckenhoupt weights, II.  General weights'',
{\it Ann. Acad. Sci. Fenn. Math.}, {\bf 36}:1 (2011), 111–138.

\bibitem{haroske3} D.D. Haroske, L. Skrzypczak, ``Entropy numbers of embeddings of
function spaces with Muckenhoupt weights, III. Some limiting
cases,'' {\it J. Funct. Spaces Appl.} {\bf 9}:2 (2011), 129–178.



\bibitem{kolm_tikh1} A.N. Kolmogorov, V.M. Tikhomirov, ``$\varepsilon$-entropy and $\varepsilon$-capacity
of sets in function spaces'' (Russian) {\it Uspehi Mat. Nauk},
{\bf 14}:2(86) (1959), 3--86.

\bibitem{kuhn_01_g} T. K\"{u}hn, ``Entropy numbers of diagonal operators of logarithmic
type'', {\it Georgian Math. J.} {\bf 8}:2 (2001), 307-318.

\bibitem{kuhn_01} T. K\"{u}hn, ``A lower estimate for entropy
numbers'', {\it J. Appr. Theory}, {\bf 110} (2001), 120--124.

\bibitem{kuhn_05} T. K\"{u}hn, ``Entropy numbers of general diagonal
operators'', {\it Rev. Mat. Complut.}, {\bf 18}:2 (2005),
479--491.

\bibitem{kuhn4} Th. K\"{u}hn, H.-G. Leopold, W. Sickel, and L. Skrzypczak. ``Entropy numbers of embeddings
of weighted Besov spaces'', {\it Constr. Approx.}, {\bf 23}
(2006), 61–77.

\bibitem{kuhn_leopold} Th. K\"{u}hn, H.-G. Leopold, W. Sickel, L.
Skrzypczak, ``Entropy numbers of embeddings of weighted Besov
spaces II'', {\it Proc. Edinburgh Math. Soc.} (2) {\bf 49} (2006),
331--359.

\bibitem{kuhn_tr_mian} T. K\"{u}hn, ``Entropy Numbers in Weighted Function Spaces.
The Case of Intermediate Weights'', {\it Proc. Steklov Inst.
Math.}, {\bf 255} (2006), 159--168.



\bibitem{kuhn5} Th. K\"{u}hn, H.-G. Leopold, W. Sickel, and L. Skrzypczak, ``Entropy numbers of embeddings
of weighted Besov spaces III. Weights of logarithmic type'', {\it
Math. Z.}, {\bf 255}:1 (2007), 1–15.

\bibitem{kuhn_08} T. K\"{u}hn, ``Entropy numbers in sequence spaces with an application to weighted
function spaces'', {\it J. Appr. Theory}, {\bf 153} (2008),
40--52.


\bibitem{lifs_m} M.A. Lifshits, ``Bounds for entropy numbers for some critical
operators'', {\it Trans. Amer. Math. Soc.}, {\bf 364}:4 (2012),
1797–1813.

\bibitem{l_l} M.A. Lifshits, W. Linde, ``Compactness properties of weighted summation operators
on trees'', {\it Studia Math.}, {\bf 202}:1 (2011), 17--47.

\bibitem{l_l1} M.A. Lifshits, W. Linde, ``Compactness properties of weighted summation operators
on trees --- the critical case'', {\it Studia Math.}, {\bf 206}:1
(2011), 75--96.

\bibitem{lif_linde} M.A. Lifshits, W. Linde, ``Approximation and entropy numbers of Volterra operators
with application to Brownian motion'', {\it Mem. Amer. Math.
Soc.}, {\bf 157}:745, Amer. Math. Soc., Providence, RI, 2002.

\bibitem{step_lom} E.N. Lomakina, V.D. Stepanov, ``Asymptotic estimates for the approximation and entropy
numbers of the one-weight Riemann–Liouville operator'', {\it Mat.
Tr.}, {\bf 9}:1 (2006), 52–100 [{\it Siberian Adv. Math.}, {\bf
17}:1 (2007), 1–36].

\bibitem{p_mattila} P. Mattila, {\it Geometry of sets
and measures in Euclidean spaces}. Cambridge Univ. Press, 1995.


\bibitem{mieth_15} T. Mieth, ``Entropy and approximation numbers of embeddings of
weighted Sobolev spaces'', {\it J. Appr. Theory}, {\bf 192}
(2015), 250--272.


\bibitem{piet_op} A. Pietsch, {\it Operator ideals}. Mathematische Monographien
[Mathematical Monographs], 16. Berlin, 1978. 451 pp.

\bibitem{c_schutt} C. Sch\"{u}tt, ``Entropy numbers of diagonal operators between symmetric
Banach spaces'', {\it J. Appr. Theory}, {\bf 40} (1984), 121--128.

\bibitem{tikh_entr} V.M. Tikhomirov, ``The $\varepsilon$-entropy of certain classes of periodic functions''
(Russian) {\it Uspehi Mat. Nauk} {\bf 17}:6 (108) (1962),
163--169.

\bibitem{tr_jat} H. Triebel, ``Entropy and approximation numbers of limiting embeddings, an approach
via Hardy inequalities and quadratic forms'', {\it J. Approx.
Theory}, {\bf 164}:1 (2012), 31--46.

\bibitem{vas_width_raspr} A.A. Vasil'eva, ``Widths of function classes on sets with tree-like
structure'', {\it J. Appr. Theory}, {\bf 192} (2015), 19--59.

\bibitem{vas_entr} A.A. Vasil'eva, ``Entropy numbers of embedding operators of
weighted Sobolev spaces with weights that are functions of
distance from some $h$-set'', arXiv:1503.00144.

\bibitem{vas_har_tree} A.A. Vasil'eva, ``Estimates for norms of two-weighted summation operators
on a tree under some restrictions on weights'', {\it Math. Nachr.}
1--24 (2015) /DOI 10.1002/mana.201300355.


\bibitem{vas_john} A.A. Vasil'eva, ``Widths of weighted Sobolev classes on a John domain'',
{\it Proc. Steklov Inst. Math.}, {\bf 280} (2013), 91--119.




\bibitem{vas_vl_raspr} A.A. Vasil'eva, ``Embedding theorem for weighted Sobolev
classes on a John domain with weights that are functions of the
distance to some $h$-set'', {\it Russ. J. Math. Phys.}, {\bf 20}:3
(2013), 360--373.


\bibitem{vas_vl_raspr2} A.A. Vasil'eva, ``Embedding theorem for weighted Sobolev
classes on a John domain with weights that are functions of the
distance to some $h$-set'', {\it Russ. J. Math. Phys.} {\bf 21}:1
(2014), 112--122.



\bibitem{vas_sib} A.A. Vasil'eva, ``Some sufficient conditions for
embedding a weighted Sobolev class on a John domain'', {\it Sib.
Mat. Zh.}, {\bf 56}:1 (2015), 65--81.

\bibitem{vas_w_lim} A.A. Vasil'eva, ``Widths of weighted Sobolev classes
with weights that are functions of the distance to some $h$-set:
some limit cases'', {\it Russ. J. Math. Phys.}, {\bf 22}:1 (2015),
127--140.


\bibitem{vas_bes} A.A. Vasil'eva, ``Kolmogorov and linear
widths of the weighted Besov classes with singularity at the
origin'', {\it J. Appr. Theory}, {\bf 167} (2013), 1--41.



\end{Biblio}
\end{document}